\documentclass[11pt,reqno]{amsart}
\usepackage[utf8]{inputenc}

\usepackage[margin=1in]{geometry}
\usepackage{amsmath}
\usepackage{amsfonts}
\usepackage{amssymb}
\usepackage{amsthm}
\usepackage{enumitem}
\usepackage{mathrsfs}
\usepackage{mathtools}
\usepackage{tikz-cd}
\usepackage{enumitem}

\usepackage[colorlinks=true,hyperindex, linkcolor=magenta, pagebackref=false, citecolor=cyan,pdfpagelabels]{hyperref}

\newtheorem{theorem}{Theorem}[section]
\newtheorem{prop}[theorem]{Proposition}
\newtheorem{lemma}[theorem]{Lemma}
\newtheorem{cor}[theorem]{Corollary}

\theoremstyle{definition}
\newtheorem{definition}[theorem]{Definition}
\newtheorem{example}[theorem]{Example}
\newtheorem{remark}[theorem]{Remark}

\newtheorem{question}[theorem]{Question}

\renewcommand{\rm}{\mathrm}

\newcommand{\Spec}{\mathrm{Spec} \,}

\newcommand{\eps}{\epsilon}

\newcommand{\vp}{\varphi}

\newcommand{\bA}{{\mathbf A}}

\newcommand{\bF}{{\mathbf F}}
\newcommand{\bG}{{\mathbf G}}

\newcommand{\bZ}{{\mathbf Z}}

\newcommand{\sA}{{\mathscr A}}

\newcommand{\sD}{{\mathscr D}}

\newcommand{\sH}{{\mathscr H}}

\newcommand{\sM}{{\mathscr M}}

\newcommand{\sO}{{\mathscr O}}
\newcommand{\sP}{{\mathscr P}}

\newcommand{\sR}{{\mathscr R}}

\newcommand{\sT}{{\mathscr T}}

\newcommand{\fm}{\mathfrak{m}}

\newcommand{\fp}{\mathfrak{p}}
\newcommand{\fq}{\mathfrak{q}}

\newcommand{\bGa}{\bG_{\rm{a}}}
\newcommand{\bGm}{\bG_{\rm{m}}}

\DeclareMathOperator{\Ad}{Ad}
\DeclareMathOperator{\AffSch}{AffSch}
\DeclareMathOperator{\Ann}{Ann}

\DeclareMathOperator{\Ass}{Ass}
\DeclareMathOperator{\Aut}{Aut}
\DeclareMathOperator{\chara}{char}

\DeclareMathOperator{\diag}{diag}

\DeclareMathOperator{\Frac}{Frac}

\DeclareMathOperator{\gl}{\mathfrak{gl}}
\DeclareMathOperator{\GL}{GL}

\DeclareMathOperator{\Hom}{Hom}

\DeclareMathOperator{\id}{id}

\DeclareMathOperator{\Isom}{Isom}
\DeclareMathOperator{\Lie}{Lie}

\DeclareMathOperator{\Sets}{Sets}
\DeclareMathOperator{\SL}{SL}
\DeclareMathOperator{\SO}{SO}
\DeclareMathOperator{\Sp}{Sp}

\DeclareMathOperator{\uAut}{\underline{\Aut}}
\DeclareMathOperator{\uHom}{\underline{\Hom}}
\DeclareMathOperator{\uIsom}{\underline{\Isom}}

\newcommand{\wtil}[1]{\wtil}

\newcommand{\ov}[1]{\overline{#1}}

\newcommand{\sch}[1]{#1\textrm{-}\rm{sch}}

\newcommand{\gp}[1]{#1\textrm{-}\rm{gp}}
\newcommand{\ratgp}[1]{#1\textrm{-}\rm{rat.gp}}

\title{Hom schemes for algebraic groups}
\author{Sean Cotner}

\begin{document}
\bibliographystyle{halpha-abbrv}

\begin{abstract}
In SGA3, Demazure and Grothendieck showed that if $G$ and $H$ are smooth affine group schemes over a scheme $S$ and $G$ is reductive, then the functor of $S$-homomorphism $G \to H$ is representable. In this paper we extend this result to cover cases in which $G$ is not reductive, with much simpler proofs. Our results apply in particular to parabolics over any base, and they are essentially optimal over a field. We also relate the closed orbits in Hom schemes to Serre's theory of complete reducibility, answer a question of Furter--Kraft, and provide many examples.
\end{abstract}

\maketitle

\section{Introduction}

\subsection{Groups over a field}

One way to express the ``rigidity" of algebraic geometry, as compared to other fields of geometry, is to note that moduli spaces of objects in algebraic geometry are often themselves algebraic varieties. For instance, if $X$ and $Y$ are projective varieties over a field $k$, then Grothendieck's work on the Hilbert scheme shows that the functor $\uHom_{\sch{k}}(X, Y)$ of $k$-scheme morphisms $X \to Y$ is a disjoint union of quasi-projective $k$-schemes. However, this rigidity depends in an essential way on the compactness of the objects involved. By way of contrast, if $X$ and $Y$ are instead \textit{affine}, then $\uHom_{\sch{k}}(X, Y)$ is almost never representable.\footnote{By standard principles of ``spreading out" and \cite[IV\textsubscript{3}, 8.14.2]{EGA}, if $\sH = \uHom_{\sch{k}}(X, Y)$ is representable then it is locally of finite type. But if $X$ is affine and $Y$ is smooth then the tangent space of $\sH$ at a $k$-morphism $f$ is isomorphic to $\rm{Hom}_{\sO_X}(f^*(\Omega^1_{Y/k}), \sO_X)$ by \cite[III, 5.3]{SGA1}, and this is almost never finite-dimensional over $k$.} In the absence of compactness, group structures can often serve a ``rigidifying" role in geometry (cf.\ Hilbert's fifth problem). Thus if $G$ and $H$ are algebraic groups over $k$ then it is reasonable to ask whether the functor $\uHom_{\gp{k}}(G, H)$ of $k$-\textit{group} homomorphisms is representable. This question has been studied previously in \cite{SGA3III}, \cite{Faisceaux-amples}, \cite{CGP}, \cite{Furter-Kraft}, \cite{Brion-hom}, and \cite{booher-tang}.

We will see in Examples~\ref{example:ga-quotient-0}, \ref{example:ga-quotient-p}, and \ref{example:ga-conn-comp} that, if $G^0$ admits a surjective $k$-homomorphism to $\bGa$, then Artin's axioms \cite[5.4]{Artin-stacks} can fail for $\uHom_{\gp{k}}(G, H)$, precluding representability. However, the following theorem shows that this is the only thing that can go wrong.

\begin{theorem}\label{theorem:intro-field}
    Let $k$ be a field, and let $G$ and $H$ be finite type $k$-group schemes. If there is no surjective $k$-homomorphism $G^0 \to \bGa$, then the functor $\uHom_{\gp{k}}(G, H)$ is representable by a disjoint union of quasi-projective $k$-schemes.
\end{theorem}

For a more precise statement, see Theorem~\ref{theorem:field-rep}. For a (partial) generalization to the case of an artin local base, see Corollary~\ref{cor:artin-rep}. Several cases of Theorem~\ref{theorem:intro-field} were already known:
\begin{enumerate}
    \item If $G$ is a torus, this is due to Grothendieck \cite[XI, 4.2]{SGA3II} and Raynaud \cite[IX, 2.6]{Faisceaux-amples}.
    \item If $G$ is connected reductive, this is (essentially) due to Demazure \cite[XXIV, 7.2.3]{SGA3III}.
    \item If $G$ is \textit{anti-affine}, then this is due to Brion \cite[Theorem 1]{Brion-hom}.
    \item When $\chara k = 0$, the restriction of the functor $\uHom_{\gp{k}}(G, H)$ to the category of \textit{reduced} $k$-schemes is shown to be an affine variety in some cases by Furter--Kraft in \cite[8.4.1]{Furter-Kraft}.\footnote{For more precise results in characteristic $0$, see Proposition~\ref{prop:char-0-hom-scheme}.}
\end{enumerate}
Note that (1) and (2) are addressed in \cite{SGA3III} and \cite{Faisceaux-amples} over a general base scheme; we will discuss this more in Section~\ref{ss:intro-base}. In the case that $G$ is reductive, our proof of Theorem~\ref{theorem:field-rep} is \textit{much} simpler than the proof of \cite[XXIV, 7.2.3]{SGA3III}, owing to a few elementary observations about groups generated by tori. See Section~\ref{section:outline} for more details.

Note that $H$ acts on $\uHom_{\gp{k}}(G, H)$ by conjugation. In general, the orbits of this action are difficult to describe: when $H = \GL_n$, the $H$-orbits correspond to the isomorphism classes of $n$-dimensional representations of $G$, and it is famously intractable to describe these when $\chara k > 0$. However, we can characterize the \textit{closed} orbits. Recall from \cite{Serre-cr} that if $k$ is an algebraically closed field and $G$ is a closed $k$-subgroup scheme of the connected reductive group $H$, then $G$ is $H$\textit{-completely reducible} ($H$-cr) if for every parabolic $k$-subgroup $P \subset H$ containing $G$, there is a Levi $k$-subgroup $L \subset P$ containing $G$. The first part of the following theorem was proven when $G$ is reductive in \cite[2.11]{booher-tang}, and the result is slightly improved by Theorem~\ref{theorem:closed-orbit}.

\begin{theorem}\label{theorem:intro-closed-orbits}
    Let $k$ be a field, and let $G$ and $H$ be finite type affine $k$-group schemes. Suppose there is no surjective $k$-homomorphism $G^0 \to \bGa$ and $H$ is connected reductive. A $k$-homomorphism $f\colon G \to H$ has closed $H$-orbit in $\uHom_{\gp{k}}(G, H)$ if and only if $f(G_{\ov{k}})$ is $H_{\ov{k}}$-cr. If $G$ is smooth and $k$ is perfect, then every component of $\uHom_{\gp{k}}(G, H)$ with a $k$-point contains a unique closed $H$-orbit.
\end{theorem}

In \cite[3.7]{BMR05}, it is shown that if $k = \ov{k}$ and $x_1, \dots, x_n \in G(k)$ topologically generate $G$ (in the sense that $G = \ov{\langle x_1, \dots, x_n \rangle}$) and $G$ is a closed $k$-subgroup of the connected reductive $H$, then $G$ is $H$-cr if and only if the $H$-orbit of $(x_1, \dots, x_n)$ in $H^n$ is closed. We view Theorem~\ref{theorem:intro-closed-orbits} as giving a more natural characterization of complete reducibility, for at least three reasons:
\begin{enumerate}
    \item Theorem~\ref{theorem:intro-closed-orbits} applies to non-smooth group schemes $G$.
    \item \cite[3.7]{BMR05} depends on an (arbitrary) auxiliary choice of topological generators for $G$.
    \item Although $\uHom_{\gp{k}}(G, H)$ is usually not of finite type when $\chara k > 0$, each component \textit{is} of finite type, and the components often have finitely many $H$-orbits. (By contrast, $H^n$ \textit{never} has finitely many $H$-orbits unless $H$ is finite or $n = 0$.)
\end{enumerate}
One caveat is that, whereas $H^n$ is smooth, this is often not true of $\uHom_{\gp{k}}(G, H)$ (see \cite[2.8]{booher-tang}). In fact, even when $G$ and $H$ are reductive we do not have very many positive results concerning the geometry of $\uHom_{\gp{k}}(G, H)$ in general. I hope to address this in future work.

We also prove several complements: in Section~\ref{section:isomorphisms} we study Isom schemes, and in Section~\ref{section:char-0} we prove a few extra results when $\chara k = 0$. In Appendix~\ref{section:weak-generation}, we answer a question posed in \cite{Furter-Kraft}.

\subsection{Groups over a base}\label{ss:intro-base}

We now move on to considering Hom schemes for \textit{families} of algebraic groups. In other words, if $S$ is \textit{any} scheme and $G$ and $H$ are $S$-group schemes, then we are interested in the representability of $\uHom_{\gp{S}}(G, H)$. This question is considerably more subtle than the question considered in Theorem~\ref{theorem:intro-field}: Examples~\ref{example:semi-abelian}, \ref{example:comp-gp-jump}, \ref{example:purity}, and \ref{example:heisenberg} show that $\sH = \uHom_{\gp{S}}(G, H)$ can fail to be representable even if $\sH_{S'}$ is representable for every artin local $S$-scheme $S'$.

If $G$ is smooth and $S$-affine, then we will say that $G$ \textit{admits an open cell} if there exists a fiberwise maximal $S$-torus $T \subset G$ and $T$-stable closed $S$-subgroup schemes $U_i \subset G$, $1 \leq i \leq n$, such that
\begin{enumerate}
    \item $U_i \cong (\bGa)_S$ for all $i$ and $T$ acts nontrivially on $U_i$,
    \item the multiplication morphism $T \times_S \prod_{i=1}^n U_i \to G$ is a fiberwise dense open embedding.
\end{enumerate}
The fpqc-local existence of an open cell implies that $G$ has connected fibers and fibral Cartan subgroups of $G$ are tori. It seems plausible that the converse holds as well, but we do not prove that. In any case, split reductive group schemes and their parabolics admit open cells, and these are the cases of most interest to us.

\begin{theorem}\label{theorem:intro-global}
    Let $S$ be a scheme, and let $G$ and $H$ be smooth $S$-affine $S$-group schemes. Suppose that there is an fpqc cover $\{S'_i \to S\}$ such that each $G_{S'_i}$ admits an open cell. Then $\uHom_{\gp{S}}(G, H)$ is representable by a separated $S$-scheme, locally of finite presentation. If $S'_i$ is normal and locally noetherian, then $\uHom_{\gp{S'_i}}(G_{S'_i}, H_{S'_i})$ is a disjoint union of $S'_i$-affine $S'_i$-schemes.
\end{theorem}

For a more precise statement, see Theorem~\ref{theorem:global-rep}. If $G$ is a reductive $S$-group scheme, then Theorem~\ref{theorem:intro-global} is proven in \cite[XXIV, 7.2.3]{SGA3III} (apart from the affineness claim). In \cite[2.2]{booher-tang}, this is improved to the case that $G^0$ is a reductive $S$-group scheme, and we give a similar improvement in Theorem~\ref{theorem:global-rep}. Our proof is considerably different and, in our opinion, much simpler than the one in \cite{SGA3III}. One key technical input which was not available at the time of \cite{SGA3III} is the notion of purity from \cite{Raynaud-Gruson} and its relation to module-freeness of coordinate rings.

We conclude this section by remarking that the hypotheses in Theorem~\ref{theorem:intro-global} are not optimal. If $G$ arises as the base change of a smooth affine group scheme $G_0$ over an artin local ring, then the result remains true as long as the special fiber of $G_0^0$ admits no nontrivial homomorphisms to $\bGa$; see Corollary~\ref{cor:artin-rep}. Example~\ref{example:moy-prasad} shows that representability can hold even if the rank of $G$ is not locally constant, but Example~\ref{example:heisenberg} shows that some care is needed if one wishes to formulate (and prove) a more general result.

\subsection{Outline of the proofs}\label{section:outline}

In \cite[XXIV, 7.2.3]{SGA3III}, Demazure shows that if $G$ and $H$ are smooth $S$-affine $S$-group schemes such that $G$ is reductive, then $\uHom_{\gp{S}}(G, H)$ is representable by a separated $S$-scheme, locally of finite presentation. The proof works, roughly speaking, as follows.
\begin{enumerate}
    \item Consider the case that $G$ is a torus \cite[XI, 4.2]{SGA3II}.
    \item Prove that maximal tori exist in $G$ Zariski-locally on $S$ \cite[XIV, 3.20]{SGA3II}. Using this, a descent argument\footnote{This step does not appear explicitly in the proof of \cite[XXIV, 7.2.3]{SGA3III}, but it seems to be an accidental omission; \textit{some} descent argument is required to reduce to the case that $G$ admits a maximal torus, and even \'etale descent is not sufficient because no significant geometric properties of $\uHom_{\gp{S}}(G, H)$ are established in \cite{SGA3III}.} reduces one to proving that if $G$ admits a \textit{split} maximal torus $T$, then the restriction morphism $\uHom_{\gp{S}}(G, H) \to \uHom_{\gp{S}}(T, H)$ is affine.
    \item Using a pinning for $(G, T)$, realize $\uHom_{\gp{S}}(G, H)$ as a closed subscheme of the scheme $\uHom_{\gp{S}}(T, H) \times_S H^N$ for some non-negative integer $N$ (depending on $G$).
\end{enumerate}
None of these steps is straightforward, but step (3) is by far the most complicated; the hardest part is to establish the ``theorem of generators and relations" \cite[XXIII, 3.5.1]{SGA3III}, whose proof relies on long computations which check every rank $2$ root system separately. Apart from being tedious, this step also makes the method of proof inapplicable to more general $G$.

Our proofs of Theorems~\ref{theorem:intro-field} and \ref{theorem:intro-global} also use step (1) above, but the rest of the argument is very different. Let us first consider Theorem~\ref{theorem:intro-field}. Arguments with ``anti-affine" group schemes as in \cite{Brion-hom} reduce us to the case that $G$ is smooth affine connected. An elementary observation (Lemma~\ref{lemma:gen-by-tori-field}) shows that there exist finitely many $k$-subtori $T_1, \dots, T_n \subset G$ such that the multiplication morphism $T_1 \times \cdots \times T_n \to G$ is dominant. This implies that there is a natural monomorphism $\iota\colon \uHom_{\gp{k}}(G, H) \to \prod_{i=1}^n \uHom_{\gp{k}}(T_i, H)$ given by restriction, and the key point is to show that $\iota$ is a closed embedding. This is shown in Lemma~\ref{lemma:strong-generation-by-subschemes}, whose proof relies principally on fpqc descent and the fact that rational group homomorphisms are defined everywhere.\footnote{In order to accommodate the proof of Theorem~\ref{theorem:intro-global}, Sections~\ref{section:generalities} and \ref{section:hom-functors} use the notion of purity from \cite{Raynaud-Gruson}, but this input is not necessary over a field, in which case all of the proofs simplify considerably.}

The idea behind the proof of Theorem~\ref{theorem:intro-global} is very similar, but the details are more complicated. A descent argument (much simpler than (2) above, using only our affineness claim) allows us to pass to arbitrary fpqc covers of $S$. In particular, deformation theory for tori shows that we can localize to find $S$-subtori $T_1, \dots, T_n$ of $G$ such that the multiplication morphism $\mu\colon T_1 \times_S \cdots \times_S T_n \to G$ is universally schematically dominant. However, our method is \textit{not} sufficient to show that the map
\[
\uHom_{\gp{S}}(G, H) \to \prod_{i=1}^n \uHom_{\gp{S}}(T_i, H)
\]
is a closed embedding in general. This is not a defect of the method: $\uHom_{\gp{S}}(G, H)$ can fail to be representable without strong assumptions like those in Theorem~\ref{theorem:intro-global} (see Example~\ref{example:heisenberg}).

To remedy the issue in the previous paragraph, one must choose the tori $T_1, \dots, T_n$ more carefully, and the hypotheses in Theorem~\ref{theorem:intro-global} are designed to make this task manageable. The key technical issue is to find a fiberwise dense open subscheme $\Omega \subset G$ such that the fiber product $\mu^{-1}(\Omega) \times_\Omega \mu^{-1}(\Omega)$ is \textit{pure} in the sense of \cite{Raynaud-Gruson}. We give a very general representability criterion in Theorem~\ref{theorem:general-criterion}, and in Sections~\ref{section:generic} and \ref{section:global} we verify this criterion in some cases to prove generic and global representability results. In order to accommodate further examples in Section~\ref{section:examples}, our criterion is more flexible than is required for the main results.

\subsection{Notation and conventions}\label{section:notation}

If $S$ is a local scheme, then we will always use $s$ to denote its closed point; if $S$ is an irreducible scheme, then we will always use $\eta$ to denote its generic point.

If $S$ is a scheme and $s \in S$ is a point, then $\ov{s}$ denotes a geometric point of $S$ factoring through $s$.

If $S$ is a scheme and $X, S'$ are $S$-schemes, then $X_{S'}$ will always denote the fiber product $X \times_S S'$.

We use the word ``embedding" wherever \cite{EGA} would use the word ``immersion".

If $X$ is an affine scheme over a ring $R$, then $R[X]$ denotes the coordinate ring of $X$.

A scheme is \textit{qcqs} if it is quasi-compact and quasi-separated.

If $S$ is a scheme of characteristic $p$ and $X$ is an $S$-scheme, then $F_{X/S}^n\colon X \to X^{(p^n)}$ is the relative $p^n$-Frobenius morphism.

If $k$ is a ring, $k'$ is a finite $k$-algebra which is free as a $k$-module and $X'$ is an affine $k'$-scheme, then $\mathrm{R}_{k'/k} X'$ is the Weil restriction, an affine $k$-scheme defined functorially on $k$-algebras $A$ by $\mathrm{R}_{k'/k} X'(A) = X'(A \otimes_k k')$.

Suppose that $S$ is a scheme and $P$ is a property of morphisms of $S$-schemes which is preserved by base change. If $f\colon F \to G$ is a natural transformations of set-valued functors on the category of $S$-schemes, then we will say that $f$ \textit{is representable by $P$-morphisms} if for every $S$-scheme $X$ and every element of $G(X)$, the fiber product $F \times_G X$ is representable and the projection $F \times_G X \to X$ satisfies $P$. We will also say that $f$ \textit{is} $P$ or \textit{is a $P$-morphism} if it is representable by $P$-morphisms, with a single exception in the case that $P$ is the property ``separated"; we say that $f$ is \textit{separated} if the relative diagonal $F \to F \times_G F$ is a closed embedding.

\subsection{Acknowledgements}

Work on this paper was originally motivated by joint work with Jeremy Booher and Shiang Tang on Galois deformation theory \cite{booher-tang}, and I thank them for a pleasant and fruitful collaboration. I thank Michel Brion, Brian Conrad, and Benjamin Martin for helpful comments and questions which improved this paper. Finally, I thank the referee for a close reading of this paper, which I believe substantially improved its clarity. This material is partially based upon work supported by the National Science Foundation under Award No.\ 2402231.

\section{Generalities}\label{section:generalities}

In this section we collect several basic results that we will need. Section~\ref{section:purity} collects results on the notion of purity from \cite{Raynaud-Gruson} which we could not find stated explicitly in the literature. For the field-minded reader, it is certainly overkill; Lemma~\ref{lemma:main-purity-result} implies that all schemes over a field are pure. Section~\ref{section:groups} contains a few elementary observations about algebraic groups, as well as some discussion of Chevalley's theorem and anti-affine groups.

\subsection{Purity}\label{section:purity}

We will use the notion of purity from \cite[3.3.3]{Raynaud-Gruson}. Recall first that if $X \to S$ is a morphism of schemes and $\sM$ is a quasi-coherent $\sO_X$-module, then the \textit{relative assassin} of $\sM$ in $X$ relative to $S$ is the set
\[
\Ass_{X/S}(\sM) = \bigcup_{s \in S} \Ass_{X_s}(\sM_s),
\]
where $\Ass_{X_s}(\sM_s)$ is the set of associated points of $\sM_s$ in $X_s$. We write $\Ass(X/S) = \Ass_{X/S}(\sO_X)$.

If $X \to S$ is a locally finite type morphism of schemes and $\sM$ is a finite type quasicoherent $\sO_X$-module, then $\sM$ is $S$-\textit{pure} if for every $s \in S$, if $(\widetilde{S}, \widetilde{s})$ is a henselization of $(S, s)$ then every point of $\Ass_{X_{\widetilde{S}}}(\sM_{\widetilde{S}})$ admits a specialization in $X_{\widetilde{s}}$. We say that $X$ is $S$-pure if $\sO_X$ is $S$-pure.

We begin with the following elementary lemma.

\begin{lemma}\label{lemma:associated-specialization}
    Let $A$ be a DVR, let $X$ be a locally finite type $A$-scheme, and let $\sM$ be a coherent $\sO_X$-module which is flat over $A$. If $y \in \Ass_{X_\eta}(\sM_\eta)$ admits a specialization in $X_s$, then it admits a specialization to some $x \in \Ass_{X_s}(\sM_s)$.
\end{lemma}

\begin{proof}
    We may and do reduce to the case $X = \Spec B$ for a finite type $A$-algebra $B$ and $\sM$ corresponds to a finite $B$-module $M$. Let $y = \fq \in \Spec B_\eta$. Note that by definition there is some $m \in M_\eta$ such that $\fq = \Ann_{B_\eta}(m)$. Since $M$ is $A$-flat, we have $M \subset M_\eta$, and by multiplying $m$ by some power of a uniformizer of $A$ we may and do assume $m \in M$ and the image $\ov{m}$ of $m$ in $M_s$ is nonzero. Since $M$ is $A$-pure, the image $I$ of $\fq \cap B$ in $B_s$ is nonzero. Thus $I \subset \Ann_{B_s}(\ov{m})$ and it follows from \cite[6.1(i)]{Matsumura} that there is some element $m_0 \in M_s$ such that $\fp = \Ann_{B_s}(m_0)$ is a prime ideal of $B_s$ containing $I$. But then $\fp \in \Ass_{B_s}(M_s)$ specializes $\fq$, so we win.
\end{proof}

The following lemma generalizes part of \cite[2.1.7]{Romagny-effective}. The final statement also provides an intuitive view on the meaning of purity.

\begin{lemma}\label{lemma:main-purity-result}
    Let $S$ be a scheme, let $X \to S$ be a finitely presented morphism, and let $\sM$ be a finitely presented quasi-coherent\footnote{Since $X$ is not assumed to be locally noetherian, this does not imply that $\sM$ is coherent.} $\sO_X$-module which is $S$-flat. The following are equivalent.
    \begin{enumerate}
        \item $\sM$ is $S$-pure,
        \item there exists an \'etale cover $\{V_i \to S\}_{i \in I}$ with $V_i$ affine, and finite affine open covers $\{U_j\}_{j \in J_i}$ of $X_{V_i}$ such that $\Gamma(U_j, \sM)$ is a free $\Gamma(V_i, \sO_{V_i})$-module for all $i \in I$ and $j \in J_i$.\footnote{It is shown in \cite[3.3.5]{Raynaud-Gruson} that if $X$ and $S$ are affine then $\sM$ is $S$-pure if and only if $\Gamma(X, \sM)$ is a projective $\Gamma(S, \sO_S)$-module. The reason for the weaker statement here is that if $S$ is affine but $X$ is not, then it is not clear that there exists an affine open cover of $X$ on which the restriction of $\sM$ is $S$-pure.}
    \end{enumerate}
    If $S$ is locally noetherian and $\sM = \sO_X$, then this is equivalent to the condition that, for every map $\Spec R \to S$ with $R$ a DVR, every associated point of $X_\eta$ has Zariski closure in $X_R$ with nonempty special fiber.
\end{lemma}

\begin{proof}
    To see that (1) implies (2), note that \cite[3.3.13]{Raynaud-Gruson} shows that for each $s \in S$ and each $x \in X_s$, if $(\widetilde{S}, \widetilde{s})$ is a henselization of $(S, s)$ then there is an affine open neighborhood $\widetilde{U}$ of $x$ in $X_{\widetilde{S}}$ such that $\Gamma(\widetilde{U}, \sM)$ is a free $\Gamma(\widetilde{S}, \sO)$-module. By \cite[3.3.5]{Raynaud-Gruson}, it follows that $\sM|_{\widetilde{U}}$ is $\widetilde{S}$-pure. By definition of the henselization, we have $\widetilde{S} = \varprojlim_{(V, v)} V$, where $(V, v)$ ranges over all residually trivial affine \'etale neighborhoods of $(S, s)$. Since $X$ and $\sM$ are of finite presentation, a standard limit argument (see the proof of Lemma~\ref{lemma:ind-scheme}) shows that there exists a residually trivial affine \'etale neighborhood $(V, v) \to (S, s)$ and an affine open neighborhood $U$ of $x$ in $X_V$ such that $U \times_V \widetilde{S} = \widetilde{U}$. By \cite[IV\textsubscript{3}, 11.2.6(ii)]{EGA} and \cite[3.3.10]{Raynaud-Gruson} (which also use the finite presentation hypotheses on $X$ and $\sM$), we may choose $(V, v)$ such that $\sM|_U$ is $V$-flat and $V$-pure. Thus by \cite[3.3.5]{Raynaud-Gruson} again, we find that $\Gamma(U, \sM)$ is $\Gamma(V, \sO)$-free, as desired. The backward direction follows similarly, using principally \cite[3.3.5, 3.3.7]{Raynaud-Gruson}.
    
    The final condition in the lemma is implied for henselian $R$ by the prior ones by the definition of purity and its preservation under base change \cite[3.3.7]{Raynaud-Gruson}; for general $R$, we use the fact that if $R$ is a DVR and $y \in X_\eta$ is an associated point, then its image in the generic fiber of $X_{R^h}$ is also associated. Since the formation of schematic closures commutes with flat base change, the claim for $R$ follows from the claim for $R^h$. The converse is immediate from the definition of purity, Lemma~\ref{lemma:associated-specialization}, and \cite[II, 7.1.9]{EGA}.
\end{proof}

\begin{definition}
    Let $X$ be a scheme. Recall from \cite[IV\textsubscript{3}, 11.10.2]{EGA} that a family $\{f_i\colon Z_i \to X\}_{i \in I}$ of morphisms is \textit{schematically dominant} if for every open $U \subset X$, the induced homomorphism
    \[
    \Gamma(U, \sO_{X}) \to \prod_{i \in I} \Gamma(f_i^{-1}(U), \sO_{Z_i})
    \]
    is injective. Following \cite[3.1.1]{Romagny-effective}, we say that $\{f_i\}$ is \textit{weakly schematically dominant} if there is no proper closed subscheme of $X$ through which each $f_i$ factors. If each $f_i$ is a locally closed embedding, then we say that $\{Z_i\}$ is \textit{(weakly) schematically dense} if $\{f_i\}$ is (weakly) schematically dominant.

    If $X$ is an $S$-scheme, then a family $\{f_i\}$ as above is \textit{universally (weakly) schematically dominant} (relative to $S$) if $\{(f_i)_{S'}\}$ is (weakly) schematically dominant for every $S$-scheme $S'$. The term \textit{universally (weakly) schematically dense} is defined similarly.
\end{definition}

\begin{lemma}\label{lemma:purity-open}
    Let $S$ be a scheme, and let $X \to S$ be a finitely presented pure flat morphism. If $\Omega \subset X$ is a universally schematically dense finitely presented open subscheme, then $\Omega$ is $S$-pure.
\end{lemma}

\begin{proof}
    We may and do work Zariski-locally on $S$ to assume $S = \Spec A$ is affine. Writing $A = \varinjlim_i A_i$ for finite type $\bZ$-algebras $A_i$, we may apply \cite[3.3.7, 3.3.10]{Raynaud-Gruson} and \cite[IV\textsubscript{3}, 11.2.6(ii)]{EGA} to ``spread out" from $A$ to some $A_i$ and thus assume $S$ is locally noetherian. By \cite[IV\textsubscript{3}, 11.10.10]{EGA}, the open $\Omega$ contains $\Ass(X/S)$, so the result follows from Lemma~\ref{lemma:associated-specialization} and the last statement of Lemma~\ref{lemma:main-purity-result}.
\end{proof}

\subsection{Groups generated by tori}\label{section:groups}

Recall the following definition from \cite[2.4.3]{CGP}.

\begin{definition}
    Let $k$ be a field and let $G$ be a finite type affine $k$-group scheme. We say that $G$ is \textit{generated by tori} over $k$ if there exist $k$-subtori $T_1, \dots, T_n \subset G$ such that the multiplication morphism $T_1 \times \cdots \times T_n \to G$ is schematically dominant.
\end{definition}

As in \cite[A.2.11]{CGP}, if $G$ is a smooth $k$-group scheme then we denote by $G_{\rm{t}}$ the largest $k$-subgroup of $G$ generated by tori.

\begin{lemma}\label{lemma:gen-by-tori-field}
    Let $k$ be a field and let $G$ be a smooth connected affine $k$-group scheme. The following are equivalent:
    \begin{enumerate}
        \item $G$ is generated by tori,
        \item every $k$-homomorphism $G \to \bGa$ is constant.
    \end{enumerate}
\end{lemma}

\begin{proof}
    If $f\colon G \to \bGa$ is a $k$-homomorphism, then $f|_T$ is trivial for every $k$-subtorus $T \subset G$, so (1) implies (2). Moreover, \cite[A.2.11]{CGP} shows that the quotient $G/G_{\rm{t}}$ is unipotent, so if (1) fails then (2) fails by \cite[IV, 2.2.1]{DG}.
\end{proof}

We recall two fundamental short exact sequences. Let $k$ be a field, and let $G$ be a smooth connected $k$-group scheme. By \cite[\S III.3.8]{DG}, there is a short exact sequence
\begin{equation}\label{equation:rosenlicht}
1 \to G_{\rm{ant}} \to G \to A(G) \to 1
\end{equation}
in which $G_{\rm{ant}}$ is \textit{anti-affine} (i.e., $\Gamma(G_{\rm{ant}}, \sO_{G_{\rm{ant}}}) = k$; see \cite{Brion-anti-affine}) and $A(G)$ is a linear algebraic group. Moreover, by Chevalley's theorem \cite[1.1]{Conrad-Chevalley}, if $k$ is perfect then there is another short exact sequence
\begin{equation}\label{equation:chevalley}
1 \to G_{\rm{aff}} \to G \to B(G) \to 1
\end{equation}
in which $G_{\rm{aff}}$ is a linear algebraic group and $B(G)$ is an abelian variety. We will maintain this notation for the following two lemmas.

\begin{lemma}\label{lemma:surjection-affine-parts}
    Let $k$ be a perfect field, and let $G$ be a smooth connected $k$-group scheme. The natural map $\pi\colon G_{\rm{aff}} \to A(G)$ is surjective and $G = G_{\rm{aff}} \cdot G_{\rm{ant}}$.
\end{lemma}

\begin{proof}
    The two conclusions are evidently equivalent, and the latter is proven in \cite[3.1]{Brion-anti-affine}.
\end{proof}

\begin{lemma}\label{lemma:gen-by-tori-surj}
    Let $k$ be a field, and let $f\colon G \to H$ be a surjective $k$-homomorphism of smooth connected $k$-group schemes. The restriction $f_{\rm{t}}\colon G_{\rm{t}} \to H_{\rm{t}}$ is also surjective.
\end{lemma}

\begin{proof}
    We may and do assume that $k$ is perfect. Let $T_1, \dots, T_n$ be $k$-subtori of $H$ for which the multiplication morphism $T_1 \times \cdots \times T_n \to H_{\rm{t}}$ is dominant. To check that $f_{\rm{t}}$ is surjective, it is enough to show that it is dominant by \cite[VI\textsubscript{B}, 1.2]{SGA3I}. Passing from $(G, H)$ to $(f^{-1}(T_i), T_i)$ for all $i$, we may and do assume that $H$ is a torus. The map $G_{\rm{ant}} \to H$ is clearly trivial, so $G \to H$ factors through $A(G)$. Since $G_{\rm{aff}} \to A(G)$ is surjective by Lemma~\ref{lemma:surjection-affine-parts}, we see that $G_{\rm{aff}} \to H$ is a surjective homomorphism of linear algebraic groups. We conclude by \cite[11.14]{Borel}.
\end{proof}

In order to clarify the role of ``topological finite generation" employed in \cite{BMR05}, we prove one lemma. We will say that a finite type $k$-group scheme $G$ is \textit{topologically finitely generated} (tfg) if there exists a field extension $K/k$ and elements $g_1, \dots, g_n \in G_K(K)$ such that there is no proper closed $K$-subgroup scheme of $G_K$ containing every $g_i$. Note that tfg group schemes are automatically smooth.

\begin{lemma}\label{lemma:tfg}
    Let $k$ be a field, and let $G$ be a smooth finite type $k$-group scheme. If $\chara k = 0$, then $G$ is tfg. If $\chara k > 0$, then the following are equivalent.
    \begin{enumerate}
        \item $G$ is tfg,
        \item there is no surjective $k$-homomorphism $G^0 \to \bGa$.
    \end{enumerate}
\end{lemma}

\begin{proof}
    If $\chara k = 0$, this is proven in \cite[Lemma 3]{Brion-Aut-End}, so we will assume $\chara k > 0$. First, both conditions are stable under field extensions: this is clear for the first condition. For the second, note that the formation of $(G^0)_{\rm{t}}$ commutes with arbitrary field extensions of $k$ \cite[A.2.11]{CGP}, so the claim follows from Lemma~\ref{lemma:gen-by-tori-field}. Thus in this proof we may assume that $k$ is algebraically closed (and in particular perfect) and not the algebraic closure of a finite field.
    
    It is clear that if $1 \to N \to E \to Q \to 1$ is a short exact sequence of smooth finite type $k$-group schemes such that $N$ and $Q$ are tfg, then $E$ is also tfg. Conversely, if $E$ is tfg then the same is true of $Q$, and if $Q$ is moreover finite then $N$ is tfg. In view of the short exact sequence (\ref{equation:rosenlicht}) applied to $G^0$, we may therefore reduce to the case that $G$ is either finite, affine and connected, or anti-affine. If $G$ is finite then the result is evident. If $G$ is anti-affine, then $G$ is either trivial or admits an abelian variety quotient by (\ref{equation:chevalley}), so we may assume that there is some $g \in G(k)$ of infinite order (since $k$ is not an algebraic extension of a finite field). In particular, $N = \ov{\langle g \rangle}$ is positive-dimensional. Since $G$ is commutative \cite[III, \S 3, 8.3]{DG}, we conclude that $N$ is normal in $G$. The quotient $G/N$ is also anti-affine, so the result follows by dimension induction. Thus we may and do assume that $G$ is affine and connected.
    
    Since tori are tfg (as follows from the fact that $\bGm$ has non-torsion elements over sufficiently large fields), it follows that $G_{\rm{t}}$ is tfg. By \cite[A.2.11]{CGP}, the quotient $G/G_{\rm{t}}$ is unipotent. Thus (2) implies (1), and we may and do assume that $G = \bGa$. But then every finitely generated subgroup of $G(k)$ is finite, so $G$ is not tfg.
\end{proof}

\section{Hom functors}\label{section:hom-functors}

In this section, we introduce our notation for Hom functors and collect basic results. Section~\ref{section:rel-rep} contains all of the basic results on relative representability that we need. Section~\ref{section:criteria} applies these relative representability results in order to obtain a very general criterion for representability of $\uHom_{\gp{S}}(G, H)$. Finally, Section~\ref{section:hom-proper} corrects the proof of \cite[1.3(2)]{Martin-Vinberg}, puts it in a slightly larger context, and proves the main lemma needed for Theorem~\ref{theorem:intro-closed-orbits}. The field-minded reader is advised to assume that $S = \Spec k$ in every result of this section, and to ignore the word ``pure".

\subsection{Relative representability}\label{section:rel-rep}

If $S$ is a scheme and $X$ and $Y$ are $S$-schemes, define the functor $\uHom_{\sch{S}}(X, Y)$ to send an $S$-scheme $S'$ to the set of $S'$-morphisms $X_{S'} \to Y_{S'}$. In this section we establish a few basic results on relative representability.

\begin{lemma}\label{lemma:demazure-gabriel}
    Let $S$ be a scheme, and let $X$ be an $S$-pure flat $S$-scheme of finite presentation. If $i\colon Y \to Z$ is a closed embedding of finitely presented $S$-schemes, then the natural morphism
    \[
    i_*\colon \uHom_{\sch{S}}(X, Y) \to \uHom_{\sch{S}}(X, Z)
    \]
    is a closed embedding of finite presentation. If $i$ is moreover open, then $i_*$ is a clopen embedding.
\end{lemma}

\begin{proof}
    By Lemma~\ref{lemma:main-purity-result} and descent, we may pass to an \'etale cover of $S$ to assume that $S$ is affine and there exists a finite affine open cover $\{U_i\}_{i \in I}$ of $X$ such that $\bigoplus_{i \in I} \Gamma(U_i, \sO)$ is a free $\Gamma(S, \sO_S)$-module. In this case, the fact that $i_*$ is a closed embedding follows from \cite[I, 2.7.5]{DG}. Note that since $Y$ and $Z$ are of finite presentation and $X$ is qcqs, \cite[IV\textsubscript{3}, 8.8.2(i)]{EGA} shows that if $(S_\lambda)$ is any inverse system of $S$-affine $S$-schemes with inverse limit $S'$, then the natural map
    \[
    \varinjlim_\lambda \uHom_{\sch{S}}(X, Y)(S_\lambda) \to \uHom_{\sch{S}}(X, Y)(S')
    \]
    is bijective. It follows formally that $i_*$ commutes with inverse limits of $\uHom_{\sch{S}}(X, Z)$-affine $\uHom_{\sch{S}}(X, Z)$-schemes, so \cite[IV\textsubscript{3}, 8.14.2]{EGA} shows that $i_*$ is locally of finite presentation (see the definition in Section~\ref{section:notation}). Being a closed embedding, it is therefore of finite presentation.
    
    For the second claim, note that the infinitesimal criterion for \'etale morphisms is immediately verified when $i$ is an open embedding. By \cite[IV\textsubscript{4}, 17.9.1]{EGA} (and the definition in Section~\ref{section:notation}), it follows that $i_*$ is an open embedding.
\end{proof}

As in \cite{Brion-hom}, we use this lemma to deduce the following corollary. Note that if $\vp\colon X' \to X$ is an $S$-morphism, then there is an induced morphism $\vp^*\colon \uHom_{\sch{S}}(X, Y) \to \uHom_{\sch{S}}(X', Y)$.

\begin{lemma}\label{lemma:brion}
    Let $S$ be a scheme, and let $X$ and $Y$ be $S$-schemes with $Y$ separated.
    \begin{enumerate}
        \item\label{item:brion-1} If $\vp_1, \vp_2\colon X' \to X$ are $S$-morphisms such that $X'$ is $S$-pure, flat, and finitely presented, then the equalizer $\ker(\vp_1^*, \vp_2^*) \to \uHom_{\sch{S}}(X, Y)$ is a closed embedding.
        \item\label{item:brion-2} If $f\colon X' \to X$ is an fpqc $S$-morphism for which $X' \times_X X'$ is $S$-pure, flat, and finitely presented, then $f^*\colon \uHom_{\sch{S}}(X, Y) \to \uHom_{\sch{S}}(X', Y)$ is a closed embedding.
        \item\label{item:brion-3} If $X$ is $S$-pure, flat, and finitely presented, then $\uHom_{\sch{S}}(X, Y)$ is separated over $S$.
    \end{enumerate}
\end{lemma}

\begin{proof}
    The first two points are proved in precisely the same way as \cite[2.2, 2.3]{Brion-hom}, but we give the proofs for completeness. Note that there is a Cartesian diagram
    \[
    \begin{tikzcd}
        \ker(\vp_1^*, \vp_2^*) \arrow[r] \arrow[d]
            &\uHom_{\sch{S}}(X, Y) \arrow[d] \\
        \uHom_{\sch{S}}(X', Y) \arrow[r, "\Delta_*"]
            &\uHom_{\sch{S}}(X', Y \times_S Y),
    \end{tikzcd}
    \]
    where $\Delta\colon Y \to Y \times_S Y$ is the diagonal. Since $Y$ is separated by assumption, Lemma~\ref{lemma:demazure-gabriel} shows that $\Delta_*$ is a closed embedding, so (\ref{item:brion-1}) follows by base change.
    
    For (\ref{item:brion-2}), note that if $\pi_1, \pi_2 \colon X' \times_X X' \to X'$ are the two projections, then by fpqc descent we have $\uHom_{\sch{S}}(X, Y) = \ker(\pi_1^*, \pi_2^*)$. Thus (\ref{item:brion-2}) follows from (\ref{item:brion-1}).
    
    Point (\ref{item:brion-3}) follows from applying Lemma~\ref{lemma:demazure-gabriel} to the closed embedding $i = \Delta_{Y/S}\colon Y \to Y \times_S Y$.
\end{proof}

If $G$ is an $S$-group scheme, then we use $m_G\colon G \times_S G \to G$ to denote the multiplication morphism. If $\Omega \subset G$ is a fiberwise dense open subscheme and $H$ is another $S$-group scheme, define the functor $\uHom_{\ratgp{S}}(\Omega, H)$ on $S$-schemes to be the functor of ``rational $S$-group homomorphisms". In other words, if $S'$ is an $S$-scheme then $\uHom_{\ratgp{S}}(\Omega, H)(S')$ is the set of $S'$-morphisms $f\colon \Omega_{S'} \to H_{S'}$ for which the diagram
\[
\begin{tikzcd}
    (\Omega_{S'} \times_{S'} \Omega_{S'}) \cap m_G^{-1}(\Omega)_{S'} \arrow[r, "f \times f"] \arrow[d, "m_G"]
        & H_{S'} \times_{S'} H_{S'} \arrow[d, "m_H"] \\
    \Omega_{S'} \arrow[r, "f"]
        &H_{S'}
\end{tikzcd}
\]
commutes. By \cite[XVIII, 2.3(i)]{SGA3II}, the restriction map $\uHom_{\gp{S}}(G, H) \to \uHom_{\ratgp{S}}(\Omega, H)$ is an isomorphism, so this will only serve as useful notation.

\begin{lemma}\label{lemma:gp-sub-rep}
    Let $S$ be a scheme, and let $G$ and $H$ be $S$-group schemes such that $G$ is fppf and $S$-pure and $H$ is separated. If $\Omega \subset G$ is a fiberwise dense open subscheme, then the restriction map $\uHom_{\gp{S}}(G, H) \to \uHom_{\sch{S}}(\Omega, H)$ is representable by closed embeddings.
\end{lemma}

\begin{proof}
    By Lemma~\ref{lemma:purity-open}, the $S$-scheme $\Omega$ is fppf and pure. Since $\uHom_{\gp{S}}(G, H) = \uHom_{\ratgp{S}}(\Omega, H)$, the lemma follows directly from the definitions and Lemma~\ref{lemma:brion}(\ref{item:brion-1}), applied to the maps $f \circ m_G$ and $m_H \circ (f \times f)$ as in the above diagram.
\end{proof}

The following lemma is a first ``representability" result. Recall that a functor $X\colon \AffSch_{/S} \to \Sets$ is a \textit{(strict) ind-scheme} if there is a presentation $X = \varinjlim_{i \in I} X_i$ of $X$ as a filtered colimit of $S$-schemes $X_i$ with transition maps which are closed embeddings. If $P$ is a property of schemes, then we say that $X$ is \textit{ind-}$P$ if one can choose such a presentation where each $X_i$ is $P$.

\begin{lemma}\label{lemma:ind-scheme}
    Let $S = \Spec A$ be an affine scheme, and let $X$ and $Y$ be pure flat finitely presented affine $S$-schemes. The functor $\uHom_{\sch{S}}(X, Y)$ is an ind-(finitely presented) ind-affine ind-scheme. In particular, if $X$ and $Y$ are $S$-group schemes then $\uHom_{\gp{S}}(X, Y)$ is an ind-(finitely presented) ind-affine ind-scheme.
\end{lemma}

\begin{proof}
    We begin by reducing to the case that $S$ is noetherian; the following ``spreading out" argument is ``standard" but not often explained outside of \cite[IV]{EGA}, so we explain it once at the request of the referee. Note that $A = \varinjlim_i A_i$ for a direct system of finitely presented $\bZ$-algebras $A_i$. Let $S_i = \Spec A_i$. Since $X$ and $Y$ are finitely presented affine $A$-schemes, there exists some $i$ and finitely presented affine $A_i$-schemes $X_i$ and $Y_i$ such that $X_i \times_{S_i} S \cong X$ and $Y_i \times_{S_i} S \cong Y$. Indeed, if $X = \Spec A[t_1, \dots, t_n]/(f_1, \dots, f_m)$ for formal variables $t_j$ and $f_k \in A[t_1, \dots, t_n]$, then since only finitely many coefficients occur in each of the finitely many $f_k$, we may find some $i$ and $f_{i1}, \dots, f_{im} \in A_i[t_1, \dots, t_n]$ with images $f_1, \dots, f_m$ in $A[t_1, \dots, t_n]$. We may then take $X_i = \Spec A_i[t_1, \dots, t_n]/(f_{i1}, \dots, f_{im})$, and similarly for $Y_i$, after perhaps passing to a larger $i$. If $X$ and $Y$ are $S$-group schemes, then by \cite[IV\textsubscript{3}, 8.8.2(i)]{EGA} (applied to the multiplication, inversion, and identity section maps, as well as the commutative diagrams in the axioms of a group scheme) we may pass to a larger $i$ to assume that $X_i$ and $Y_i$ are also $S_i$-group schemes. By \cite[IV\textsubscript{3}, 11.2.6(ii)]{EGA} and \cite[3.3.10]{Raynaud-Gruson}, we may pass to a yet larger $i$ to assume that $X_i$ and $Y_i$ are $S_i$-flat and $S_i$-pure. Since $\uHom_{\sch{S}}(X, Y) \cong \uHom_{\sch{S_i}}(X_i, Y_i) \times_{S_i} S$ and similarly for $\uHom_{\gp{S}}$, we may pass from $S$ to $S_i$ in what follows to assume that $S$ is noetherian.
    
    By \cite[3.3.5]{Raynaud-Gruson}, we know that $\Gamma(X, \sO_X)$ and $\Gamma(Y, \sO_Y)$ are both free $A$-modules. Fix a finite free $A$-submodule $V \subset \Gamma(Y, \sO_Y)$ which generates $\Gamma(Y, \sO_Y)$ as an $A$-algebra. Let $\{f_i\}_{i \in I}$ be an $A$-module basis for $\Gamma(X, \sO_X)$. For each finite subset $J \subset I$, let $W_J$ be the $A$-submodule generated by $\{e_j\}_{j \in J}$ and let $\sH_J$ denote the subfunctor of $\uHom_{\sch{S}}(X, Y)$ consisting of those $\vp\colon X \to Y$ for which $\vp^{\#}(V) \subset W_J$. It is evident that $\uHom_{\sch{S}}(X, Y)$ is the filtered colimit of the $W_J$, and it is standard to check that each $\sH_J$ is representable by an affine $S$-scheme of finite presentation and each map $\sH_J \to \sH_K$ is a closed embedding (when $J \subset K$); see for example the proof of \cite[Lemma A.8.13]{CGP}. The final claim follows from Lemma~\ref{lemma:gp-sub-rep}.
\end{proof}

\begin{remark}
    It follows from Lemma~\ref{lemma:ind-scheme} that, if $\uHom_{\gp{S}}(X, Y)$ is representable, then every quasi-compact open subscheme $U$ is quasi-affine. Indeed, Lemma~\ref{lemma:ind-scheme} shows that we have $U = \varinjlim_i U_i$ for finitely presented quasi-affine $S$-schemes $U_i$, and \cite[1.22]{Timo-ind} shows that $U = U_i$ for some $i$. Moreover, \cite[1.4]{Timo-ind} shows that $\uHom_{\gp{S}}(X, Y)$ is a scheme if and only if it is an algebraic space, so proving Theorems~\ref{theorem:intro-field} and \ref{theorem:intro-global} amounts to verifying Artin's axioms. However, we have found it difficult to verify these directly beyond the case that $X$ is a torus.
\end{remark}

The key point in the proofs of Theorems~\ref{theorem:intro-field} and \ref{theorem:intro-global} is that one can often prove representability of $\uHom_{\gp{S}}(G, H)$ by reducing it to representability of $\uHom_{\gp{S}}(X_i, H)$ for a collection of $S$-subgroup schemes $X_i$ of $G$ which ``generate" $G$ in some sense. To clarify the relevant notion of ``generate", we begin with the following definition.

\begin{definition}\label{def:generation}
    If $S$ is a scheme and $G$ is a finitely presented flat $S$-group scheme, we will say that a collection of $S$-scheme morphisms $\{f_i\colon X_i \to G\}_{i \in I}$ from finitely presented flat $S$-schemes $X_i$ \textit{weakly generates} $G$ if for every $S$-scheme $S'$ there are no proper closed $S'$-subgroup schemes of $G_{S'}$ through which each $(f_i)_{S'}$ factors. Equivalently, let $\Sigma$ denote the set of all finite sequences of elements of $I$ (possibly with repetition). For every sequence $\sigma = (i_1, \dots, i_m) \in \Sigma$, let $X_\sigma = X_{i_1} \times_S \cdots \times_S X_{i_m}$, and note that there is a morphism $f_\sigma\colon X_\sigma \to G$ induced by the $f_i$ and the multiplication morphism for $G$. Then $\{f_i\}_{i \in I}$ weakly generates $G$ if and only if $\{f_\sigma\}_{\sigma \in \Sigma}$ is weakly schematically dominant.
    
    We will say that $\{f_i\}_{i \in I}$ \textit{strongly generates} $G$ if each $X_i$ is $S$-pure and there is some sequence $\sigma \in \Sigma$ and a fiberwise dense open subscheme $\Omega \subset G$ such that $f_\sigma^{-1}(\Omega) \to \Omega$ is faithfully flat and $f_\sigma^{-1}(\Omega) \times_{\Omega} f_\sigma^{-1}(\Omega)$ is pure over $S$.
\end{definition}

By \cite[VI\textsubscript{B}, 7.4]{SGA3I} and generic flatness, if $k$ is a field, each $X_i$ is geometrically reduced and geometrically connected, and $1 \in X_i(k)$ for all $i$, then $\{X_i\}$ weakly generates $G$ if and only if it strongly generates $G$. The conditions in the definition of ``strongly generates" may appear contrived, but they are justified by the proof of the following lemma. This lemma is the main observation which makes our representability proofs easier than the one in \cite[XXIV]{SGA3III}.

\begin{lemma}\label{lemma:strong-generation-by-subschemes}
    Let $S$ be a scheme, and let $G$ and $H$ be $S$-group schemes such that $G$ is flat, finitely presented, and pure, and $H$ is separated. Let $\{f_i\colon X_i \to G\}_{i \in I}$ strongly generate $G$. The natural morphism $\uHom_{\gp{S}}(G, H) \to \prod_{i \in I} \uHom_{\sch{S}}(X_i, H)$ is a closed embedding.
\end{lemma}

\begin{proof}
    Let $(i_1, \dots, i_n) \in I^n$ and $\Omega \subset G$ be as in the definition of strong generation, so the natural multiplication morphism $\mu\colon X_{i_1} \times_S \cdots \times_S X_{i_n} \to G$ is such that $\mu^{-1}(\Omega) \to \Omega$ is faithfully flat and $\mu^{-1}(\Omega) \times_\Omega \mu^{-1}(\Omega)$ is $S$-pure. By Lemma~\ref{lemma:brion}(\ref{item:brion-2}) and Lemma~\ref{lemma:gp-sub-rep}, the natural map $\iota\colon \uHom_{\gp{S}}(G, H) \to \uHom_{\sch{S}}(\mu^{-1}(\Omega), H)$ is a closed embedding. If $m\colon \prod_{j=1}^n \uHom_{\sch{S}}(X_{i_j}, H) \to \uHom_{\sch{S}}(\mu^{-1}(\Omega), H)$ is the morphism given by sending a tuple $(f_j)$ of morphisms to the composition
    \[
    X_{i_1} \times_S \cdots \times_S X_{i_n} \xrightarrow{\prod_{j=1}^n f_j} H^n \xrightarrow{m_H} H
    \]
    restricted to $\mu^{-1}(\Omega)$, then $\iota$ factors through the map 
    \[
    h\colon \prod_{i \in I} \uHom_{\sch{S}}(X_i, H) \xrightarrow{\pi} \prod_{j=1}^n \uHom_{\sch{S}}(X_{i_j}, H) \xrightarrow{m} \uHom_{\sch{S}}(\mu^{-1}(\Omega), H),
    \]
    where $\pi$ is the projection. By \cite[I, 2.5.6 b)]{DG}, it suffices to show that $h$ is separated.
    
    Note that every term above is separated by Lemma~\ref{lemma:brion}(\ref{item:brion-3}) and the fact that a product of separated morphisms is separated (since a product of closed embeddings of schemes is a closed embedding). If $F$ and $G$ are functors such that $F$ is separated, then any morphism $f\colon F \to G$ is separated: this follows from \cite[I, 2.5.6 b)]{DG}, using the fact that the diagonal $F \to F \times F$ factors through $F \times_G F$ and the map $F \times_G F \to F \times F$ is a monomorphism (and a fortiori separated). Thus $h$ is separated, as desired.
\end{proof}

Lemma~\ref{lemma:strong-generation-by-subschemes} fails for weak generation (see Examples~\ref{example:furter-kraft-weak} and \ref{example:furter-kraft-char-p}). However, there is a partial replacement; see Lemma~\ref{lemma:generation-by-subschemes}.

The following lemma will only be used in Example~\ref{example:sga3-endo}.

\begin{lemma}\label{lemma:identity-clopen}
    Let $S$ be a scheme, and let $G$ be an $S$-pure flat $S$-group scheme of finite presentation such that $G_s$ is generated by tori for all $s \in S$. If $H$ is a finitely presented separated $S$-group scheme, then the inclusion $e_*\colon \uHom_{\gp{S}}(G, 1) \to \uHom_{\gp{S}}(G, H)$ is a clopen embedding.
\end{lemma}

\begin{proof}
    The fact that $e_*$ is a closed embedding follows from Lemma~\ref{lemma:demazure-gabriel}. The fact that $e_*$ is locally of finite presentation follows from the fact that $G$ and $H$ are both finitely presented. By \cite[IV\textsubscript{4}, 17.9.1]{EGA}, it suffices to show that $e_*$ satisfies the infinitesimal criterion for \'etale morphisms. For this, suppose $S$ is an artin local scheme and $f\colon G \to H$ is a residually trivial $S$-homomorphism. Let $T_1, \dots, T_n$ be $S$-subtori of $G$ such that the multiplication morphism $\mu\colon T_1 \times_S \cdots \times_S T_n \to G$ has dominant special fiber. By \cite[IV\textsubscript{3}, 11.10.9]{EGA}, the morphism $\mu$ is universally schematically dominant. Moreover, by \cite[IX, 6.8]{SGA3II}, for each $i$ the morphism $f|_{T_i}$ is trivial, so each $T_i$ lies in $\ker f$ and thus $\ker f = G$, i.e., $f$ lies in the image of $e_*$.
\end{proof}

\subsection{Criteria for representability}\label{section:criteria}

In this section we aim to prove a very general representability result for $\uHom_{\gp{S}}(G, H)$ (Theorem~\ref{theorem:general-criterion}) which will be used in Section~\ref{section:representability}. Because it is so general, it will take some work to apply it to more concrete situations.

In order to apply the results on relative representability from Section~\ref{section:rel-rep}, we must have at least one \textit{absolute} representability result. Recall from \cite[\href{https://stacks.math.columbia.edu/tag/0AP6}{Tag 0AP6}]{stacks-project} that a morphism $f\colon X \to Y$ of schemes is \textit{ind-quasi-affine} if for every affine open $V \subset Y$, every quasi-compact open $U \subset f^{-1}(V)$ is quasi-affine. The main interest of this notion for our purposes is that, by \cite[\href{https://stacks.math.columbia.edu/tag/0APK}{Tag 0APK}]{stacks-project}, fpqc descent is effective for ind-quasi-affine morphisms.

The following lemma is a very mild extension of \cite[2.6]{booher-tang}, itself a mild extension of \cite[XI, 4.2]{SGA3II} and \cite[IX, 2.6]{Faisceaux-amples}.

\begin{lemma}\label{lemma:torus-source}
    Let $S$ be a scheme, let $T$ be an $S$-torus, and let $H$ be a smooth $S$-affine $S$-group scheme. Then $\underline{\Hom}_{\gp{S}}(T, H)$ is representable by a smooth $S$-ind-quasi-affine $S$-scheme. If $S_{\rm{red}}$ is normal and $S$ is either locally noetherian or qcqs, then $\underline{\Hom}_{\gp{S}}(T, H)$ is representable by a disjoint union of smooth $S$-affine $S$-schemes.
\end{lemma}

\begin{proof}
    For the first claim, we may work Zariski-locally on $S$ to assume that $S$ is affine, and a fortiori qcqs. In this case, the claim follows from the (short) proof of \cite[2.6]{booher-tang}. If moreover $S$ is normal and either locally noetherian or qcqs, then the second claim also follows from \textit{loc.\ cit.} If more generally $S_{\rm{red}}$ is normal (and either locally noetherian or qcqs), then $\uHom_{\gp{S}}(T, H)$ is a disjoint union of schemes with affine underlying reduced subscheme (because the map 
    \[
    \uHom_{\gp{S_{\rm{red}}}}(T_{S_{\rm{red}}}, H_{S_{\rm{red}}}) = \uHom_{\gp{S}}(T, H) \times_S S_{\rm{red}} \to \uHom_{\gp{S}}(T, H)
    \]
    is a thickening by base change), which are hence affine by \cite[II, 1.6.4]{EGA}.
\end{proof}

We are now ready to prove the desired general representability result.

\begin{theorem}\label{theorem:general-criterion}
    Let $S$ be a scheme, and let $G$ be a flat pure finitely presented $S$-group scheme. Suppose we are given
    \begin{itemize}
        \item $S$-subtori $T_1, \dots, T_n \subset G$,
        \item an integer $N \geq 1$ and flat pure finitely presented $S$-subschemes $X_1, \dots, X_m \subset \left(\prod_{i=1}^n T_i\right)^N$.
    \end{itemize}
    Suppose that $G$ is strongly generated by $T_1, \dots, T_n, X_1, \dots, X_m$. Then for every smooth $S$-affine $S$-group scheme $H$, the functor $\uHom_{\gp{S}}(G, H)$ is representable by an $S$-ind-quasi-affine $S$-scheme, locally of finite presentation. If $S_{\rm{red}}$ is normal and either locally noetherian or qcqs, then $\uHom_{\gp{S}}(G, H)$ is a disjoint union of finitely presented $S$-affine $S$-schemes.
\end{theorem}

\begin{proof}
    Let $\sT = \prod_{i=1}^n \uHom_{\gp{S}}(T_i, H)$. By Lemma~\ref{lemma:strong-generation-by-subschemes}, the restriction morphism
    \[
    r\colon \uHom_{\gp{S}}(G, H) \to \sT \times_S \prod_{j=1}^m \uHom_{\sch{S}}(X_j, H)
    \]
    is representable by closed embeddings. Note that $r$ factors through the morphism $F\colon \sT \to \sT \times_S \prod_{j=1}^m \uHom_{\sch{S}}(X_j, H)$ given by
    \[
    (f_1, \dots, f_n) \mapsto \left(f_1, \dots, f_n, \left(\prod_{i=1}^n f_i\right)^N|_{X_1}, \dots, \left(\prod_{i=1}^n f_i\right)^N|_{X_m}\right).
    \]
    Note that $F$ is monic, hence separated, so \cite[I, 2.5.6 b)]{DG} shows that $\uHom_{\gp{S}}(G, H) \to \sT$ is a closed embedding, and the result follows from Lemma~\ref{lemma:torus-source}.
\end{proof}

\begin{remark}
    In the proofs of the main results of this paper, we will use Theorem~\ref{theorem:general-criterion} only when $m = 0$; in other words, we will have no need of the $X_j$. See however Example~\ref{example:moy-prasad} for an illustration of the full strength of the theorem.
\end{remark}

\subsection{Hom-properness}\label{section:hom-proper}

The only result in this section which is used in the sequel is Lemma~\ref{lemma:finite-git}. We introduce the following definition mainly to clarify the role of reductivity in Theorem~\ref{theorem:intro-closed-orbits}. See Remark~\ref{remark:hom-proper-def} below for some justification of its conditions.

\begin{definition}
    Let $S$ be a locally noetherian scheme, and let $G$ be a smooth $S$-affine $S$-group scheme. We will say that $G$ is \textit{Hom-proper} if the quotient stack $[\uHom_{\gp{S'}}(H, G_{S'})/G_{S'}]$ satisfies the existence part of the valuative criterion of properness for all locally noetherian $S$-schemes $S'$ and all smooth $S'$-affine $S'$-group schemes $H$.

    Concretely, this means that if $R$ is a DVR with fraction field $K$, we have a morphism $\Spec R \to S$ and a smooth affine $R$-group scheme $H$, and $f_1\colon H_K \to G_K$ is a $K$-homomorphism, then there exists an extension $R \subset R'$ of DVRs with $K' = \Frac(R')$ and some $g \in G(K')$ such that $g(f_1)_{K'} g^{-1}$ extends to an $R'$-homomorphism $f\colon H_{R'} \to G_{R'}$.
\end{definition}

Recall \cite[9.1.1, 9.7.5, 9.7.6]{Alper-adequate} (see also point (1) in \cite[\S 2]{Martin-Vinberg}) that if $S$ is a scheme then a smooth $S$-affine $S$-group scheme is \textit{geometrically reductive} if $G^0$ is a reductive group scheme and $G/G^0$ is finite.

Recall that if $R$ is a DVR with valuation $v$ and fraction field $K$ and $X$ is an affine $K$-scheme, then a subset $B \subset X(K)$ is \textit{bounded} if for every $f \in K[X]$, the function $b \mapsto v(f(b))$ is bounded below on $B$.

\begin{lemma}\label{lemma:4.3-correction}
    Let $R$ be an excellent henselian DVR with fraction field $K$, and let $G$ be a geometrically reductive smooth affine $R$-group scheme. If $B \subset G(K)$ is a bounded subgroup, then there exists a finite local extension $R \subset R'$ with fraction field $K'$ and a geometrically reductive smooth affine $R'$-integral model $G'$ of $G_{K'}$ such that $B \subset G'(R')$. Moreover, for any such $G'$, one may pass to a further extension of $R'$ to assume that the subgroup $G'(R')$ is $G^0(K')$-conjugate to $G(R')$.
\end{lemma}

Lemma~\ref{lemma:4.3-correction} is stated in \cite[4.3]{Martin-Vinberg}, but the proof given there of the final claim is not complete unless the smooth $S$-affine geometrically reductive $S$-group scheme $G$ has connected fibers; otherwise the claim made in the paragraph preceding the lemma statement that $G(K)^1$ (in the notation of \textit{loc.\ cit.}) contains every bounded subgroup of $G(K)$ can fail.\footnote{This can be seen already in the case that $G$ is the normalizer of a maximal torus in $\SL_2$: indeed, in this case every point of $G(K) - G^0(K)$ is of order $2$, so there does not exist a subgroup $X \subset G(K)$ containing every bounded subgroup and satisfying $X \cap G^0(K) = G^0(K)^1$.} Since we need (a consequence of) this lemma, we will complete the proof here.

\begin{proof}[Proof of Lemma~\ref{lemma:4.3-correction}]
    By a limit argument, we may assume that $R$ is strictly henselian. The argument given in \cite[4.3]{Martin-Vinberg} shows every claim except the final one; it also shows that, after extending $R$, the subgroup $G'^0(R)$ is $G^0(K)$-conjugate to $G^0(R)$.\footnote{Because of the error stated above, the definition of $U'$ in the first paragraph of the proof of \cite[4.3]{Martin-Vinberg} is not correct, in the sense that it can fail to contain $B$. This can be fixed by instead defining $U' = G^0(K')^1_{x'} \cdot B$, with notation as in \textit{loc.\ cit.}, after which the proof of the above claims goes through verbatim.} Thus by passing to a conjugate we may assume $G'^0(R) = G^0(R)$. By \cite[4.1]{Martin-Vinberg}, the $R$-group schemes $G^0$ and $G'^0$ are (canonically) isomorphic $R$-models of $G^0_{K}$. Let $(P, T)$ be a Borel pair of $G^0 = G'^0$. Note that $G(R) \cap N_{G(K)}(P, T)$ and $G'(R) \cap N_{G(K)}(P, T)$ are two bounded subgroups of $N_{G(K)}(P, T)$ with equal intersections with $G^0(K) = G'^0(K)$. Also, $N_{G_{K}}(P, T)$ is a smooth affine $K$-group with identity component $T_{K}$, and the conjugacy of Borel pairs (over $R$; see \cite[3.2.6, 5.2.13(2)]{Conrad}) shows that every component of $N_{G_{K}}(P, T)$ contains a point of $G(R)$ and a point of $G'(R)$. This shows $G(R) = G^0(R) \cdot N_{G(R)}(P, T)$ and similarly for $G'(R)$, so we can pass from $(G, G')$ to $(N_G(P, T), N_{G'}(P, T))$ to assume that $G^0 = G'^0$ is a torus.

    Next, note that if $m \geq 1$ is an integer, then because $G^0(K')[m] \subset G^0(R') \cap G'^0(R')$ for every $R'$ as in the statement, we may pass from $(G,G')$ to $(G/G^0[m], G'/G'^0[m])$ at will. By \cite[3.3]{Martin-Vinberg}, if $G_K/G_K^0$ is of order $n$ then there exists a finite flat $R$-subgroup scheme $E \subset G$ meeting every connected component of $G$ and satisfying $E \cap G^0 = G^0[n]$, and similarly for $G'$. Thus by passing to the quotient by $G^0[n] = G'^0[n]$ we may assume that $G$ and $G'$ are both semidirect products $G^0 \rtimes \Gamma$ (in the sense of $R$-group schemes), where $\Gamma \cong G/G^0 \cong G'/G'^0$ is a finite \'etale $R$-group scheme. Note that the action of $\Gamma$ on $G^0 = G'^0$ is the same for $G$ and $G'$, since it is determined by the action on $G^0_K = G'^0_K$.
    
    Thus $G \cong G'$ as $R$-group schemes (but perhaps not as integral models of $G_K$); composing a chosen such isomorphism with the given isomorphism $G'_K \cong G_K$ coming from the definition of integral model, we obtain a $K$-automorphism $\alpha$ of $G_K$ sending $G(R')$ to $G'(R')$ for all $R'$. Note that $\alpha$ induces a $K$-automorphism $\alpha_0$ of $G^0_K$ and an automorphism $\alpha_1$ of $\Gamma$. Since $G^0$ is an $R$-torus and $\Gamma$ is finite \'etale, $\alpha_0$ and $\alpha_1$ extend uniquely to $R$-automorphisms and thus the map $\alpha_0 \times \alpha_1\colon G = G^0 \rtimes \Gamma \to G$ is an $R$-automorphism. By composing $\alpha$ with $(\alpha_0 \times \alpha_1)^{-1}$ we may therefore assume that $\alpha$ induces the identity maps on $G^0_K$ and $\Gamma$.
    
    The map $\gamma \mapsto \alpha(\gamma)\gamma^{-1}$ is a $1$-cocycle $c\colon \Gamma \to G^0(K)$. Since $\mathrm{H}^1(\Gamma, G^0(K))$ is killed by $n$, as above we may pass from $G_K$ to $G_K/G_K^0[n]$ to assume that $c$ is a coboundary. By definition, this means that there is some $z \in G^0(K)$ such that $\alpha(\gamma)\gamma^{-1} = z\gamma z^{-1}\gamma^{-1}$ for all $\gamma \in \Gamma$, or in other words $\alpha$ is given by $z$-conjugation. Since $\alpha$ sends $G(R)$ to $G'(R)$, this means $zG(R)z^{-1} = G'(R)$, as desired.
\end{proof}

For a geometrically reductive smooth $S$-affine group scheme $G$ and any cocharacter $\lambda\colon \bGm \to G$, we have closed $S$-subgroup schemes $P_G(\lambda)$ and $Z_G(\lambda)$ of $G$ coming from the dynamic method; see \cite[\S 2.1]{CGP} for details. We call $P_G(\lambda)$ an \textit{R-parabolic subgroup scheme}, and we call $Z_G(\lambda)$ an \textit{R-Levi subgroup scheme}, as in \cite[\S 6]{BMR05}.

The following lemma gives the main examples of Hom-proper group schemes for our purposes.

\begin{lemma}\label{lemma:hom-proper}
    Let $S$ be a locally noetherian scheme.
    \begin{enumerate}
        \item\label{item:hom-proper-gr} Any geometrically reductive smooth $S$-affine $S$-group scheme is Hom-proper.
        \item\label{item:hom-proper-long} Let $U$ be a smooth $S$-affine $S$-scheme which admits a composition series consisting of $S$-group schemes isomorphic to $\bGa$, and let $G$ be a Hom-proper $S$-group scheme. Suppose that $G$ acts on $U$ and that there is a central cocharacter $\lambda\colon \bGm \to G$ such that all of the weights in the induced action of $\bGm$ on $\Lie U$ are positive. Then $G \ltimes U$ is Hom-proper.
        \item\label{item:hom-proper-par} Any R-parabolic subgroup of a geometrically reductive smooth $S$-affine $S$-group scheme is Hom-proper.
    \end{enumerate}
\end{lemma}

\begin{proof}
    Throughout, by spreading out and the valuative criterion we may and do assume $S = \Spec R$ for an excellent strictly henselian DVR $R$ with fraction field $K$. Claim (\ref{item:hom-proper-gr}) is essentially \cite[4.4]{Martin-Vinberg}; since the short proof there relies on \cite[4.3]{Martin-Vinberg} (i.e., Lemma~\ref{lemma:4.3-correction}), we recall it for completeness. Let $G$ and $H$ be smooth affine $R$-group schemes such that $G$ is geometrically reductive, and let $f_1\colon H_K \to G_K$ be a $K$-homomorphism. Note that $f_1(H(A)) \subset G(K)$ is a bounded subgroup, so by Lemma~\ref{lemma:4.3-correction} there is a finite local extension $R \subset R'$ with fraction field $K'$, a geometrically reductive smooth affine $R'$-group scheme $G'$ which is an integral model of $G_{K'}$, and an element $g \in G^0(K')$ such that $gG'(R')g^{-1} = G(R')$. By \cite[4.1]{Martin-Vinberg}, it follows that the $K'$-homomorphism $\Ad_g \circ f_1\colon H_{K'} \to G'_{K'} = G_{K'}$ extends to an $R'$-homomorphism $H_{R'} \to G_{R'}$, as desired.
    
    For (\ref{item:hom-proper-long}), it is enough by \cite[2.10.10]{Kaletha-Prasad} to show that if $H \subset G(K) \ltimes U(K)$ is a bounded subgroup, then there exists an extension $R'/R$ of DVRs with $K' = \Frac(R')$ and $g \in G(K')$ such that $gHg^{-1} \subset G(R') \ltimes U(R')$. By Hom-properness of $G$, we may extend $R$ and pass to a $G(K)$-conjugate of $H$ to assume $H \subset G(R) \ltimes U(K)$.
    
    Given a short exact sequence
    \[
    1 \to \bGa \to E \to Q \to 1
    \]
    of affine $S$-group schemes, the agreement of \'etale cohomology and coherent cohomology shows that there is a scheme-theoretic section $Q \to E$. By induction on the length of a composition series, our assumption on $U$ shows $U \cong \bA_R^N$ scheme-theoretically for some $N$. Moreover, the action of $\bGm$ on $U$ induced by $\lambda$ is diagonalizable with respect to a given such isomorphism. Since $\lambda$ is a central cocharacter with positive weights on $\Lie U$, if $\pi$ is a uniformizer of $R$ then there is some $n \geq 0$ such that $\lambda(\pi)^n H \lambda(\pi)^{-n} \subset G(R) \ltimes U(R)$, as desired. Claim (\ref{item:hom-proper-par}) follows from (\ref{item:hom-proper-gr}) and (\ref{item:hom-proper-long}).
\end{proof}

\begin{example}
    Let $k$ be a field, and let $\bGm$ act on $\bGa^2$ diagonally with weights $\pm 1$. If $G = \bGm \ltimes \bGa^2$, then $G$ is \textit{not} Hom-proper. Indeed, if $R = k[\![t]\!]$ and $K = k(\!(t)\!)$, then there exists a $K$-homomorphism $f_1\colon G_K \to G_K$ given by $f_1(a, x, y) = (a, t^{-1}x, y)$, and there is no extension $R \subset R'$ of DVRs such that $f_1$ extends to an $R'$-homomorphism after $G(K')$-conjugacy. It seems plausible that Lemma~\ref{lemma:hom-proper}(\ref{item:hom-proper-gr}) and (\ref{item:hom-proper-long}) nearly characterize Hom-proper smooth affine group schemes over a DVR.
\end{example}

\begin{remark}\label{remark:hom-proper-def}
    In the definition of Hom-properness, it may appear more natural to weaken the assumption on smoothness of $H$ to flatness. If $G$ satisfies this modified condition, call it \textit{strongly Hom-proper}. The interest in this condition is that, if a reductive $k$-group $H$ is strongly Hom-proper then the second statement of Theorem~\ref{theorem:intro-closed-orbits} extends beyond the case that $G$ is smooth.\footnote{In fact, one only needs to verify the condition of Hom-properness when $G$ is obtained via base change from a finite type group scheme over a field, which is a considerably weaker condition.}
    
    However, reductive group schemes need not be strongly Hom-proper. To see this, let $k$ be a field of characteristic $2$ and let $R = k[\![t]\!]$. Let $G = \SO_{2n+1, R}$ and let $H$ be a finite type flat $R$-group scheme with $H_\eta \cong \SO_{2n+1}$ and $(H_s)_{\rm{red}} \cong \Sp_{2n}$; such group schemes exist by \cite[\S 9]{Prasad-Yu}. Assume moreover that the normalization $\widetilde{H}$ is isomorphic to $\SO_{2n+1, R}$, as we may by \cite[9.5]{Prasad-Yu}. There is no $R$-homomorphism $f\colon H \to \widetilde{H}$ such that $f_\eta$ is an isomorphism: if there were, then the composition $\vp\colon \widetilde{H} \to H \to \widetilde{H}$ would be an $R$-homomorphism such that $\vp_\eta$ is an isomorphism but $\vp_s$ is not monic, contradicting \cite[2.14]{booher-tang}. It follows that $[\uHom_{\gp{R}}(H, \widetilde{H})/\widetilde{H}]$ does not satisfy the existence part of the valuative criterion of properness.

    I am not aware of any examples in characteristic $\neq 2$ which illustrate this phenomenon.
\end{remark}

The following lemma is the only result we need for the proof of Theorem~\ref{theorem:closed-orbit}.

\begin{lemma}\label{lemma:finite-git}
    Let $S$ be a locally noetherian scheme, and let $G$ and $H$ be smooth separated finite type $S$-group schemes such that $H$ is $S$-affine and geometrically reductive and $\uHom_{\gp{S}}(G, H)$ is a disjoint union of (finite type)\footnote{The finite type conclusion is automatic from the remaining conclusions by \cite[IV\textsubscript{3}, 8.8.2(i), 8.14.2]{EGA}.} $S$-affine $S$-schemes. Then the GIT quotient $\uHom_{\gp{S}}(G, H)/\!/H$ is a disjoint union of finite $S$-schemes.
\end{lemma}

\begin{proof}
    By Lemma~\ref{lemma:hom-proper}, the quotient stack $[\uHom_{\gp{S}}(G,H)/H]$ satisfies the existence part of the valuative criterion of properness. By \cite[9.1.4]{Alper-adequate}, the natural map $\phi\colon [\uHom_{\gp{S}}(G,H)/H] \to \uHom_{\gp{S}}(G,H)/\!/H$ is an adequate moduli space in the sense of \cite[5.1.1]{Alper-adequate}. By \cite[5.3.1]{Alper-adequate}, this implies that $\phi$ is surjective. This implies that $\uHom_{\gp{S}}(G,H)/\!/H$ satisfies the existence part of the valuative criterion of properness: indeed, given a commutative diagram
    \[
    \begin{tikzcd}
        \Spec K \arrow[r] \arrow[d]
            &\uHom_{\gp{S}}(G,H)/\!/H \arrow[d] \\
        \Spec R \arrow[r]
            &S
    \end{tikzcd}
    \]
    where $R$ is a DVR with fraction field $K$, by surjectivity of $\phi$ we may first extend $K$ to assume that there is a lift $\Spec K \to [\uHom_{\gp{S}}(G, H)/H]$. Since $[\uHom_{\gp{S}}(G, H)/H]$ satisfies the existence part of the valuative criterion of properness, after further extending $K$ we may assume that there exists a lift $\Spec R \to [\uHom_{\gp{S}}(G, H)/H]$; by composition, this furnishes a lift $\Spec R \to \uHom_{\gp{S}}(G,H)/\!/H$, as desired. Now note that $\uHom_{\gp{S}}(G,H)/\!/H$ is also a disjoint union of finitely presented $S$-affine $S$-schemes $C_i$, so each $C_i$ is a proper $S$-affine $S$-scheme, hence finite.
\end{proof}

\section{Results over a field}\label{section:field-results}

In this section, our main aim is to prove Theorems~\ref{theorem:intro-field} and \ref{theorem:intro-closed-orbits}. In Section~\ref{section:isomorphisms}, we study Isom schemes, and in Section~\ref{section:char-0} we explore phenomena which are unique to characteristic $0$.

\subsection{Representability of Hom schemes}\label{ss:field-rep}

Before proving any general results, let us illustrate several examples in which $\uHom_{\gp{k}}(G, H)$ \textit{fails} to be representable. The first two examples appear to be well-known, but I am not aware of a precise reference.

\begin{example}\label{example:ga-quotient-0}
    Let $k$ be a field of characteristic $0$ and let $G$ be a locally finite type $k$-group scheme which admits a surjective $k$-homomorphism $G \to \bGa$. If $H$ is either $\bGm$ or an abelian variety, then $\sH = \uHom_{\gp{k}}(G, H)$ is not representable. To see this, we will show that the natural map $\sH(k[\![t]\!]) \to \varprojlim_n \sH(k[t]/(t^n))$ is not surjective. By Lemma~\ref{lemma:brion}(\ref{item:brion-2}), we may and do assume $G = \bGa$. Since $\sH(k[\![t]\!]) \subset \sH(k(\!(t)\!))$ and there are no nontrivial homomorphisms $\bGa \to H$ over a field, it suffices to find a nontrivial sequence $(\vp_n)_{n \geq 1} \in \varprojlim_n \sH(k[t]/(t^n))$.
    
    For each $n \geq 0$, let $A_n = k[t]/(t^{n+1})$ and let $H_n$ denote the Weil restriction $\rm{R}_{A_n/k} H_{A_n}$. Note that $H_n$ is isomorphic to $H \times (\Lie H)^n$; let $U_n = \ker(H_n \to H) \cong (\Lie H)^n$. Smoothness of $H$ implies that the morphism $U_{n+1} \to U_n$ is surjective. Since $\chara k = 0$, any $k$-homomorphism $\bGa^m \to \bGa^n$ is linear, and it follows that any $k$-homomorphism $\bGa \to U_n$ lifts to a $k$-homomorphism $\bGa \to U_{n+1}$. Thus, using the identification
    \[
    \Hom_{\gp{A_n}}(G_{A_n}, H_{A_n}) = \Hom_{\gp{k}}(G, H_n),
    \]
    we can inductively find a nontrivial sequence $(\vp_n)_{n \geq 1} \in \varprojlim_n \sH(A_n)$.
\end{example}

\begin{example}\label{example:ga-quotient-p}
    Let $k$ be a field of characteristic $p > 0$ and let $G$ be a finite type $k$-group scheme which admits a surjective $k$-homomorphism $G \to \bGa$. We will show that if $H$ is \textit{any} positive-dimensional finite type $k$-group scheme, then $\sH = \uHom_{\gp{k}}(G, H)$ is not representable. First, using the functorial criterion \cite[IV\textsubscript{3}, 8.14.2]{EGA} and principles of spreading out, one sees that $\sH$ is locally of finite presentation if it is representable. By passing to $\ov{k}$ and replacing $H$ by $H_{\rm{red}}$ (using Lemma~\ref{lemma:demazure-gabriel}), we may and do assume $H$ is smooth. 
    
    As in Example~\ref{example:ga-quotient-0}, the Weil restriction $H' = \rm{R}_{(k[t]/(t^2))/k} H_{k[t]/(t^2)}$ is isomorphic to $H \ltimes \Lie H$. Choose a $k$-subgroup scheme $U \subset \Lie H$ and an isomorphism $U \cong \bGa$, and for each $n \geq 0$ let $\psi_n\colon \bGa \to U$ denote the homomorphism $\psi_n(x) = x^{p^n}$. Via the universal property of the Weil restriction, this sequence corresponds to a sequence of elements of the tangent space $\rm{Tan}_0 \sH$ which are $k$-linearly independent. Thus $\rm{Tan}_0 \sH$ is infinite-dimensional, precluding representability of $\sH$.
\end{example}

\begin{example}\label{example:ga-conn-comp}
    It can happen that $\uHom_{\gp{k}}(G, H)$ is representable when $G^0$ is unipotent and $H$ is not unipotent: for instance, the abelianization of $G$ may be finite (e.g., if $\chara k \neq 2$ and $G = \bZ/2 \ltimes \bGa$), and in that case $\uHom_{\gp{k}}(G, H)$ is representable whenever $H$ is commutative. However, if $\chara k > 0$ and $G^0$ admits a surjective $k$-homomorphism to $\bGa$, then there is \textit{some} $H$ for which $\uHom_{\gp{k}}(G, H)$ is not a scheme. To show this, we may and do assume $k = \ov{k}$. Moreover, using Lemma~\ref{lemma:brion}(\ref{item:brion-2}), we may and do pass to a quotient of $G$ to assume $G^0$ is a vector group. 

    By \cite[1.1]{Brion-quasi-split}, there is a finite $k$-subgroup scheme $N$ of $G$ such that $N \to G/G^0$ is surjective. Note that $N \cap G^0$ is normal in $G$, so by passing to a further quotient we may assume $G \cong \Gamma \ltimes \bGa^n$ for some finite group $\Gamma$ and some $n \geq 1$. By \cite[4.2.1]{McNinch-linearity}, we may and do pass to a further quotient of $G$ to assume that the action of $\Gamma$ on $G$ corresponds to an irreducible representation $\rho\colon \Gamma \to \GL_n(k)$.\footnote{In \cite[\S 4]{McNinch-linearity}, the group $\Gamma$ is \textit{connected}, but this hypothesis plays no role in the proof of \cite[4.2.1]{McNinch-linearity}.} In this case, we claim that $\uHom_{\gp{k}}(G, G)$ is not representable. By irreducibility, we may pass to a $\GL_n(k)$-conjugate of $\rho$ to assume that $\rho$ factors through $\GL_n(\bF_{p^N})$ for some $N$. In this case, one can check that for each $m \geq 0$, the morphism $f_m\colon G \to G$ defined by
    \[
    f_m(\gamma, x) = (\gamma, x^{p^{Nm}})
    \]
    is a $k$-homomorphism, and we preclude representability as in Example~\ref{example:ga-quotient-p}.

    If $\chara k = 0$ and $G^0$ is a vector group, then divisibility of $\bGa$ implies $G \cong \Gamma \ltimes \bGa^m$ for some finite group $\Gamma$ and representation $\rho\colon\Gamma \to \GL_m(k)$. If $\rho$ is defined over $\bZ$ (i.e., some conjugate of $\rho$ factors through $\GL_m(\bZ)$), then one can let $H = \Gamma \ltimes \bGm^m$ and show that $\uHom_{\gp{k}}(G, H)$ is not representable as in Example~\ref{example:ga-quotient-p}. In general, $\rho$ is defined over the ring of integers in some number field, but I do not know whether this can be used to build a similar example.
\end{example}

We now move to positive results. The following lemma is very similar to \cite[4.5]{Brion-hom}, but it incorporates target groups which are not of finite type.

\begin{lemma}\label{lemma:passage-to-id-comp}
    Let $k$ be a field, and let $G$ and $H$ be locally finite type $k$-group schemes. Suppose that $G$ is of finite type. The restriction morphism $r\colon \uHom_{\gp{k}}(G, H) \to \uHom_{\gp{k}}(G^0, H)$ is representable by a disjoint union of quasi-projective morphisms. If $H$ is of finite type, then $r$ is of finite type. If $H^0$ is affine, then $r$ is representable by a disjoint union of finite type affine morphisms. If $H$ is connected or affine, then $r$ is affine. Moreover, if $G^0_{\rm{red}}$ is a $k$-subgroup scheme of $G$ (e.g., if $G$ is finite or $k$ is perfect), then the same claims hold for $\ov{r}\colon \uHom_{\gp{k}}(G, H) \to \uHom_{\gp{k}}(G^0_{\rm{red}}, H)$ in place of $r$.
\end{lemma}

\begin{proof}
    Suppose first that $G$ is finite. We may write $H$ as a directed union $\bigcup_{i \in I} H_i$ for finite type clopen subschemes $H_i$ of $H$. By \cite[1.2]{Conrad-Chevalley}, each $H_i$ is quasi-projective, so the functor $\uHom_{\sch{k}}(G, H_i)$ is representable by a disjoint union of quasi-projective $k$-schemes by \cite[5.18, 5.23]{FGA-explained}. If $i \geq j$, then the map $\uHom_{\sch{k}}(G, H_j) \to \uHom_{\sch{k}}(G, H_i)$ is a clopen embedding by Lemma~\ref{lemma:demazure-gabriel}. Thus $\uHom_{\sch{k}}(G, H) = \bigcup_{i \in I} \uHom_{\sch{k}}(G, H_i)$ is representable by a disjoint union of quasi-projective $k$-schemes. By Lemma~\ref{lemma:brion}(\ref{item:brion-2}), the morphism $\uHom_{\gp{k}}(G, H) \to \uHom_{\sch{k}}(G, H)$ is representable by closed embeddings, so $\uHom_{\gp{k}}(G, H)$ is also representable by a disjoint union of quasi-projective $k$-schemes. If $H$ is of finite type, then $\uHom_{\gp{k}}(G, H)$ is of finite type: this follows from \cite[5.16]{FGA-explained} and the fact that the Hilbert polynomial of a finite subscheme of a quasi-projective scheme is equal to the order of the subscheme. Moreover, \cite[4.5]{Brion-hom} shows that $\uHom_{\gp{k}}(G, H)$ is affine if either $G$ is connected, $H$ is connected, or $H$ is affine.
    
    Suppose now that $H^0$ is affine (but perhaps $H$ is not quasi-compact). To show that $\uHom_{\gp{k}}(G, H)$ is a disjoint union of affine schemes, it suffices by \cite[2.5]{booher-tang} and spreading out to assume that $k$ is algebraically closed. If $\Gamma$ is a finite subgroup of the constant group $\pi_0(H)$, then we let $\uHom_{\gp{k}}(G, H)_{\Gamma}$ denote the subfunctor of $\uHom_{\gp{k}}(G, H)$ consisting of those homomorphisms $f$ for which the image of $\pi_0(f)$ is equal to $\Gamma$. By Lemma~\ref{lemma:demazure-gabriel}, if $q\colon H \to \pi_0(H)$ is the natural $k$-homomorphism, then $\uHom_{\gp{k}}(G, H)_{\Gamma}$ is naturally a clopen subscheme of $\uHom_{\gp{k}}(G, q^{-1}(\Gamma))$, so it is affine by the previous paragraph and the fact that $q^{-1}(\Gamma)$ is affine of finite type. It is clear that
    \[
    \uHom_{\gp{k}}(G, H) = \bigsqcup_{\Gamma \subset \pi_0(H)} \uHom_{\gp{k}}(G, H)_{\Gamma},
    \]
    and therefore $\uHom_{\gp{k}}(G, H)$ is indeed representable by a disjoint union of finite type affine $k$-schemes. We have now established the result when $G$ is finite.

    Now consider the general case. By \cite[1.1]{Brion-quasi-split}, there exists a finite $k$-subgroup scheme $N \subset G$ such that the multiplication morphism $N \ltimes G^0_{\rm{red}} \to G$ is a faithfully flat $k$-homomorphism. By Lemma~\ref{lemma:strong-generation-by-subschemes}, there is a closed embedding
    \[
    \uHom_{\gp{k}}(G, H) \to \uHom_{\gp{k}}(G^0, H) \times \uHom_{\gp{k}}(N, H),
    \]
    and similarly with $G^0_{\rm{red}}$ in place of $G^0$ when the former is a $k$-subgroup scheme. Thus the previous paragraphs allow us to conclude.
\end{proof}

The following theorem is a strengthening of Theorem~\ref{theorem:intro-field}.

\begin{theorem}\label{theorem:field-rep}
    Let $k$ be a field, and let $G$ and $H$ be locally finite type $k$-group schemes. Suppose that $G$ is of finite type and there is no surjective $k$-homomorphism $G^0 \to \bGa$. Then $\uHom_{\gp{k}}(G, H)$ is representable by a disjoint union of quasi-projective $k$-schemes.
    
    Moreover, if $G$ is connected, $H$ is connected, or $H^0$ is affine, then $\uHom_{\gp{k}}(G, H)$ is a disjoint union of finite type affine $k$-schemes.
\end{theorem}

\begin{proof}
    By Lemma~\ref{lemma:passage-to-id-comp}, we may and do assume that $G$ is connected. In this case, we have claimed that $\uHom_{\gp{k}}(G, H)$ is representable by a disjoint union of affine $k$-schemes, so by effectivity of fpqc descent for ind-quasi-affine morphisms \cite[\href{https://stacks.math.columbia.edu/tag/0APK}{Tag 0APK}]{stacks-project}, \cite[2.5]{booher-tang}, and spreading out, we may and do assume that $k$ is algebraically closed. Thus by Lemma~\ref{lemma:passage-to-id-comp} again, we may and do further assume that $G$ is smooth.

    There are two fundamental short exact sequences of smooth $k$-group schemes
    \[
    1 \to G_{\rm{ant}} \to G \to A(G) \to 1
    \]
    and
    \[
    1 \to G_{\rm{aff}} \to G \to B(G) \to 1
    \]
    as in Section~\ref{section:groups}. Let $G_{\rm{aff, t}}$ be the $k$-subgroup of $G_{\rm{aff}}$ generated by tori. By Lemma~\ref{lemma:surjection-affine-parts} the map $G_{\rm{aff}} \to A(G)$ is surjective. Note that there is also no surjective $k$-homomorphism $A(G) \to \bGa$, so by Lemma~\ref{lemma:gen-by-tori-field} and Lemma~\ref{lemma:gen-by-tori-surj}, we see that $G_{\rm{aff, t}} \to A(G)$ is surjective. Thus the multiplication morphism $G_{\rm{ant}} \times G_{\rm{aff, t}} \to G$ is surjective. Using the same argument as before, we reduce to the case that $G$ is either anti-affine or generated by tori. The first case is dealt with in \cite[Theorem 1]{Brion-hom}, and in the second case we conclude from the definitions and Theorem~\ref{theorem:general-criterion}.
\end{proof}

\begin{example}\label{example:bad-geometry}
    For general $G$ and $H$, the scheme $\sH = \uHom_{\gp{k}}(G, H)$ can have rather complicated geometry. First, the existence of the Frobenius homomorphism shows that $\sH$ is almost never quasi-compact in positive characteristic; see Example~\ref{example:sga3-endo}. For example, if $H$ is a simple $k$-group of adjoint type and $G = H \ltimes \Lie H$, then $\sH$ is not reduced. To see this, we first show that any $k$-homomorphism $f\colon G \to H$ has kernel containing $1 \times \Lie H$. For this, it is enough to show that $\Lie f\colon \Lie G \to \Lie H$ has trivial restriction to $\Lie(\Lie H)$. Indeed, if $\chara k = 0$, then this is clearly enough. In positive characteristic, we note that if $\Lie f|_{\Lie(\Lie H)}$ is trivial then $f|_{1 \times \ker(F_{\Lie H})}$ is trivial. Since $G/(1 \times \ker(F_{\Lie H})) \cong G$, it therefore follows from filtering that $f|_{1 \times \ker(F^n_{\Lie H})}$ is trivial for all $n \geq 1$, and it follows that $f|_{1 \times \Lie H}$ is trivial.
    
    To show that $\Lie f|_{\Lie(\Lie H)}$ is trivial, we note that the adjoint representation of $H$ is irreducible by simplicity of $H$, so if $\Lie f|_{\Lie(\Lie H)}$ is nontrivial then it is injective, hence surjective for dimension reasons. But then $f|_{1 \times \Lie H}$ is smooth, hence surjective, which is absurd because $\Lie H$ is a vector group and $H$ is reductive. It follows that if $S$ is any reduced $k$-scheme and $f\colon G_S \to H_S$ is an $S$-homomorphism then $1 \times \Lie H \subset \ker f$. However, there is a $k[\eps]/(\eps^2)$-homomorphism $f\colon G_{k[\eps]} \to H_{k[\eps]}$ given by 
    \[
    f(h, X) = hX
    \]
    (where we use the equality $\Lie H = \ker(\mathrm{R}_{k[\eps]/k} H_{k[\eps]} \to H)$). Thus the corresponding morphism $\Spec k[\eps] \to \sH$ cannot factor through $\sH_{\rm{red}}$, whence $\sH \neq \sH_{\rm{red}}$.

    Moreover, $\sH$ need not be locally complete intersection (lci). For an explicit example, suppose $\chara k = 2$, let $G = \bGm \ltimes \bGa$ (where $\bGm$ acts on $\bGa$ with weight 1), and let $H = \GL_4$. Consider the natural restriction morphism $\pi\colon \sH \to \uHom_{\gp{k}}(\bGm, H)$. Let $f\colon \bGm \to H$ be the homomorphism given by
    \[
    f(t) = \diag(t^3, t^2, t, 1).
    \]
    A straightforward calculation shows
    \[
    \pi^{-1}(f) \cong \Spec k[u, v, w, x, y, z]/(uv, vw, uy - z, wx - z),
    \]
    which is not lci. Since every component of $\uHom_{\gp{k}}(\bGm, H)$ is (scheme-theoretically) a single $H$-orbit, and since the stabilizer of $f$ in $\GL_4$ is the diagonal torus $T$, it follows that the natural action map $\GL_4 \times \pi^{-1}(f) \to \sH$ is a $T$-torsor onto a connected component of $\sH$. Since $\pi^{-1}(f)$ is not lci, it follows in particular that $\sH$ contains a component which is not lci.
\end{example}

I don't know any examples of connected reductive $G$ for which $\uHom_{\gp{k}}(G, H)$ fails to be reduced or lci, and it would be interesting to know whether any such examples exist.

\subsection{Isomorphism schemes}\label{section:isomorphisms}

If $S$ is a scheme and $G$ and $H$ are $S$-group schemes, define the functor of group isomorphisms $\uIsom_{\gp{S}}(G, H)$ via
\[
\uIsom_{\gp{S}}(G, H)(S') = \{f\colon G_{S'} \to H_{S'}\colon f \text{ is an } S'\text{-group isomorphism}\}
\]
The following proposition partially generalizes \cite[2.15]{booher-tang}.

\begin{prop}\label{prop:isom-scheme}
    Let $S$ be a scheme, and let $G$ and $H$ be flat finitely presented $S$-group schemes for which $\uHom_{\gp{S}}(G, H)$ and $\uHom_{\gp{S}}(H, G)$ are both representable. Then $\uIsom_{\gp{S}}(G, H)$ is representable by an open subscheme of $\uHom_{\gp{S}}(G, H)$.
\end{prop}

\begin{proof}
    There is a natural map
    \[
    c\colon \uHom_{\gp{S}}(G, H) \times_S \uHom_{\gp{S}}(H, G) \to \uHom_{\gp{S}}(G, G) \times_S \uHom_{\gp{S}}(H, H)
    \]
    given by $c(f_1, f_2) = (f_2 \circ f_1, f_1 \circ f_2)$, and $c^{-1}(\id_G, \id_H)$ represents $\uIsom_{\gp{S}}(G, H)$. Since the map $i\colon \uIsom_{\gp{S}}(G, H) \to \uHom_{\gp{S}}(G, H)$ is a monomorphism, it suffices by \cite[IV\textsubscript{4}, 17.9.1]{EGA} to show that it is \'etale. The fact that $i$ is locally of finite presentation follows from the functorial criterion \cite[IV\textsubscript{3}, 8.14.2]{EGA}, ``spreading out" (see especially \cite[IV\textsubscript{3}, 9.6.1(xi)]{EGA}), and the fibral isomorphism criterion \cite[IV\textsubscript{4}, 17.9.5]{EGA}. Moreover, the fibral isomorphism criterion shows that $i$ satisfies the infinitesimal criterion for \'etale morphisms, concluding the proof.
\end{proof}

Note that $\uIsom_{\gp{S}}(G, H)$ may fail to be \textit{closed} in $\uHom_{\gp{S}}(G, H)$: indeed, the quotient map $\bGm \ltimes \bGa \to \bGm \subset \bGm \ltimes \bGa$ is a degeneration of the identity map on $\bGm \ltimes \bGa$. See \cite[2.15]{booher-tang} for a positive result under reductivity hypotheses on $G$.

Proposition~\ref{prop:isom-scheme} may be combined with Theorem~\ref{theorem:field-rep} or the later Theorems~\ref{theorem:generic-rep} and \ref{theorem:global-rep} to yield more concrete results. In particular, we obtain the following generalization of \cite[2.4.4]{CGP}. If $G$ is an $S$-group scheme, let $\uAut_{G/S} = \uIsom_{\gp{S}}(G, G)$.

\begin{cor}
    Let $k$ be a field, and let $G$ be a finite type $k$-group scheme for which there is no surjective $k$-homomorphism $G^0 \to \bGa$. Then $\uAut_{G/k}$ is a locally finite type $k$-group scheme. If either $G$ is connected or $G^0$ is affine, then $\uAut_{G/k}^0$ is affine.
\end{cor}

\begin{proof}
    By Theorem~\ref{theorem:field-rep}, the functor $\uHom_{\gp{k}}(G, G)$ is representable, and Proposition~\ref{prop:isom-scheme} shows that $\uAut_{G/k}$ is representable by an open subscheme of $\uHom_{\gp{k}}(G, G)$, establishing the first claim. If either $G$ is connected or $G^0$ is affine, then Theorem~\ref{theorem:field-rep} shows that $\uAut_{G/k}^0$ is quasi-affine, hence affine by \cite[VI\textsubscript{B}, 12.9]{SGA3I}
\end{proof}

\subsection{Characteristic 0}\label{section:char-0}

In this section, we prove a ``representability" result for $\uHom_{\gp{k}}(G, H)$ applicable to \textit{all} finite type $G$ when $\chara k = 0$.

Recall from \cite[\href{https://stacks.math.columbia.edu/tag/0AI6}{Tag 0AI6}]{stacks-project} that an \textit{affine formal $k$-scheme} (called an \textit{affine formal algebraic space} over $k$ in \textit{loc.\ cit.}) is a functor $X\colon \AffSch_{/k} \to \Sets$ which admits a presentation as a filtered colimit $X = \varinjlim_{i \in I} X_i$ with each $X_i$ affine and with transition maps which are thickenings (i.e., closed embeddings which are bijective on points). More generally, an fppf sheaf $X\colon \AffSch_{/k} \to \Sets$ is a \textit{formal $k$-scheme} if there is a family $\{X_i\}_{i \in I}$ of open subfunctors of $X$ such that each $X_i$ is an affine formal $k$-scheme and the map $\bigsqcup_{i \in I} X_i \to X$ is an epimorphism of Zariski sheaves.

If $X$ is a formal $k$-scheme, then there is a reduced $k$-scheme $X_{\rm{red}}$ such that any $k$-morphism $f\colon U \to X$ from a reduced $k$-scheme $U$ factors uniquely through $X_{\rm{red}}$ (see \cite[\href{https://stacks.math.columbia.edu/tag/0AIN}{Tag 0AIN}]{stacks-project}). If $X$ is an affine formal $k$-scheme, then $X_{\rm{red}}$ is also affine. Explicitly, if $X = \varinjlim_i X_i$ is a presentation as in the definition, then $X_{\rm{red}} = (X_i)_{\rm{red}}$ for any $i$. Note that (the functor represented by) $X_{\rm{red}}$ is a subfunctor of $X$ which has the same values as $X$ on all \textit{reduced} $k$-algebras.

\begin{lemma}\label{lemma:unipotent-hom-scheme}
    Let $k$ be a field of characteristic $0$. Let $U$ and $H$ be locally finite type $k$-group schemes such that $U$ is unipotent. The functor $\uHom_{\gp{k}}(U, H)$ is an affine formal $k$-scheme and $\uHom_{\gp{k}}(U, H)_{\rm{red}}$ is of finite type.
\end{lemma}

\begin{proof}
    This proof is very similar to the proof of \cite[8.4.1(4)]{Furter-Kraft}, but we need to do slightly more work. Since $U$ is unipotent and $\chara k = 0$, there is an exponential map $\exp\colon \Lie U \to U$ which is an isomorphism of $k$-schemes. Moreover, if $V$ is any abelian Lie subalgebra of $\Lie U$, then $\exp|_V$ is a $k$-homomorphism. Let $V_1, \dots, V_n \subset \Lie U$ be $1$-dimensional vector subspaces which generate $\Lie U$. By Lemma~\ref{lemma:strong-generation-by-subschemes}, the restriction morphism $\uHom_{\gp{k}}(U, H) \to \prod_{i=1}^n \uHom_{\gp{k}}(V_i, H)$ is a closed embedding. Thus we may and do pass from $U$ to each $V_i$ to assume $U \cong \bGa$. We may and do further assume that $H$ is connected.

    For each $n \geq 0$, let $H_n$ be the $n$th infinitesimal neighborhood of $H_{\rm{aff}}$ in $H$. By \cite[I, 6.1.7]{EGA}, each $H_n$ is affine. Although the $H_n$ are usually not stable under multiplication, we use the notation $\uHom_{\gp{k}}(U, H_n)$ to denote the subfunctor of $\uHom_{\gp{k}}(U, H)$ consisting of homomorphisms which factor through $H_n$. We have $\uHom_{\gp{k}}(U, H) = \varinjlim_n \uHom_{\gp{k}}(U, H_n)$, where each transition map is a nilpotent thickening, so we may and do pass from $H$ to $H_{\rm{aff}}$ to assume that $H$ is affine. Choosing an embedding of $H$ inside some $\GL_n$ and applying Lemma~\ref{lemma:demazure-gabriel}, we may and do assume $H = \GL_n$.

    Write $k[\bGa] = k[t]$ and $k[\GL_n] = k[x_{ij}][\det^{-1}]$. By \cite[6.2.5]{Furter-Kraft}, every $k$-homomorphism $\lambda\colon \bGa \to \GL_n$ is of the form $\lambda(x) = \exp(xN)$ for some nilpotent matrix $N \in \gl_n$. It follows that $\lambda^{\#}(x_{ij})$ is a polynomial of degree $< n$ for all $i, j$. The proof of Lemma~\ref{lemma:ind-scheme} now shows that $\uHom_{\gp{k}}(\bGa, \GL_n) = \varinjlim_{m} X_m$, where each $X_m$ is an affine finite type $k$-scheme and the transition maps $X_m \to X_{m + 1}$ are closed embeddings, which are bijective on points as long as $m \geq n - 1$.
\end{proof}

\begin{prop}\label{prop:char-0-hom-scheme}
    Let $k$ be a field of characteristic $0$. Let $G$ and $H$ be locally finite type $k$-group schemes such that $G$ is of finite type. Then $\uHom_{\gp{k}}(G, H)$ is representable by a disjoint union of formal $k$-schemes $C$ such that $C_{\rm{red}}$ is quasi-projective. If $G$ is connected, $H$ is connected, or $H^0$ is affine, then $\uHom_{\gp{k}}(G, H)$ is a disjoint union of affine formal $k$-schemes. If $G$ is affine and there is no surjective $\ov{k}$-homomorphism $G^0_{\ov{k}} \to \bGm$, then $\uHom_{\gp{k}}(G, H)$ is an affine formal $k$-scheme.
\end{prop}

\begin{proof}
    Since $\chara k = 0$, Cartier's theorem ensures that $G$ is smooth. Lemma~\ref{lemma:passage-to-id-comp} reduces us to the case that $G$ is connected. By Lemma~\ref{lemma:surjection-affine-parts}, we have $G = G_{\rm{ant}} \cdot G_{\rm{aff}}$, so Lemma~\ref{lemma:strong-generation-by-subschemes} shows that the restriction map
    \[
    \uHom_{\gp{k}}(G, H) \to \uHom_{\gp{k}}(G_{\rm{aff}}, H) \times \uHom_{\gp{k}}(G_{\rm{ant}}, H)
    \]
    is a closed embedding, and we reduce to the case that $G$ is affine by Theorem~\ref{theorem:field-rep}. Let $U$ be the unipotent radical of $G$, so $G = G_0 \ltimes U$ for a connected reductive $k$-group $G_0$ by a theorem of Mostow \cite[7.1]{Mostow} (see also \cite[5.4.1]{Conrad}). Lemma~\ref{lemma:strong-generation-by-subschemes} ensures that the natural map
    \[
    \uHom_{\gp{k}}(G, H) \to \uHom_{\gp{k}}(G_0, H) \times \uHom_{\gp{k}}(U, H)
    \]
    is a closed embedding. We therefore obtain the second conclusion from Theorem~\ref{theorem:field-rep} and Lemma~\ref{lemma:unipotent-hom-scheme}, and we obtain the third conclusion from \cite[XXIV, 7.3.1(ii)]{SGA3III}.
\end{proof}

\begin{cor}\label{cor:char-0-reduction}
    Let $k$ be a field of characteristic $0$. Let $G$ and $H$ be locally finite type $k$-group schemes such that $G$ is of finite type. Then $\uHom_{\gp{k}}(G, H)_{\rm{red}}$ is a disjoint union of quasi-projective $k$-schemes. If $G$ is connected, $H$ is connected, or $H^0$ is affine, then $\uHom_{\gp{k}}(G, H)_{\rm{red}}$ is a disjoint union of affine $k$-schemes. If $G$ is affine and there is no surjective $\ov{k}$-homomorphism $G^0_{\ov{k}} \to \bGm$, then $\uHom_{\gp{k}}(G, H)_{\rm{red}}$ is an affine $k$-scheme.
\end{cor}

In Section~\ref{section:weak-generation}, we address a further question in characteristic $0$, posed in \cite{Furter-Kraft}.

\subsection{Closed orbits}\label{section:closed-field}

The definitions of R-parabolic and R-Levi subgroups are recalled in Section~\ref{section:hom-proper}. Recall \cite[\S 6]{BMR05} that if $k$ is a field and $H$ is a reductive $k$-group, then a closed $k$-subgroup scheme $G \subset H$ is $H$-completely reducible ($H$-cr) over $k$ if, for every R-parabolic $k$-subgroup $P \subset H$ containing $G$, there is an R-Levi $k$-subgroup $L \subset P$ containing $G$. A standard argument of ``spreading out and specialization" shows that if $K/k$ is an extension of algebraically closed fields then $G$ is $H$-cr if and only if $G_K$ is $H_K$-cr. Our main goal is to prove a slightly more precise version of Theorem~\ref{theorem:intro-closed-orbits}.

Recall that if $A$ is a ring, $X$ is a separated functor of $A$-algebras, and $f \in X(\bGm)$, then we say that \textit{$\lim_{t \to 0} f(t)$ exists} in $X$ if there is a (unique) $\widetilde{f}\in X(\bA_k^1)$ extending $f$. Moreover, we use $\lim_{t \to 0} f(t)$ to denote $\widetilde{f}(0) \in X(A)$.

\begin{lemma}\label{lemma:limits-exist}
    Let $k$ be a field, and let $G$ and $H$ be finite type $k$-group schemes such that $H$ is affine. Let $f\colon G \to H$ be a $k$-homomorphism and let $\lambda \colon\bGm \to H$ be a cocharacter.
    \begin{enumerate}
        \item $\lim_{t \to 0} \lambda(t)f\lambda(t)^{-1}$ exists in $\uHom_{\gp{k}}(G, H)$ if and only if $f(G) \subset P_H(\lambda)$.
        \item If $\lim_{t \to 0} \lambda(t)f\lambda(t)^{-1}$ exists, then the limit is equal to the composition $G \xrightarrow{f} P_H(\lambda) \to Z_H(\lambda)$.
    \end{enumerate}
\end{lemma}

\begin{proof}
    The statement that $\lim_{t \to 0} \lambda(t)f\lambda(t)^{-1}$ exists means that there is an $\bA^1_k$-homomorphism $f_\lambda\colon G \times \bA^1_k \to H \times \bA^1_k$ such that for any $k$-algebra $A$ and any $t \in \bGm(A)$, the fiber of $f_\lambda$ over $t$ is equal to $\lambda(t)f\lambda(t)^{-1}$. In this case, the limit is equal to the fiber $f_0$ of $f_\lambda$ over $0 \in \bA^1(k)$. Thus if the limit exists, we find $f_0(g) = \lim_{t \to 0} \lambda(t)f(g)\lambda(t)^{-1}$ for all $k$-algebras $A$ and all $g \in G(A)$. By definition (see \cite[paragraph following 2.1.4]{CGP}), this implies $f(g) \in P_H(\lambda)(A)$, i.e., $f$ factors through $P_H(\lambda)$.

    By definition of $P_H(\lambda)$, for each $k$-algebra $A$ and each $h \in P_H(\lambda)(A)$ there is a (unique) $A$-morphism $\vp_h\colon\bA_A^1 \to H_A$ such that $\vp_h(t) = \lambda(t) h \lambda(t)^{-1}$ for all $t \in \bGm(A)$. Define an $\bA_k^1$-homomorphism $\pi_\lambda\colon P_H(\lambda) \times \bA_k^1 \to P_H(\lambda) \times \bA_k^1$ functorially by $\pi_\lambda(h, x) = (\vp_h(x), x)$. Thus if $f$ factors through $P_H(\lambda)$ then the $\bA_k^1$-homomorphism $\pi_\lambda \circ (f \times \id_{\bA_k^1})\colon G \times \bA_k^1 \to P_H(\lambda) \times \bA_k^1$ shows that the limit $\lim_{t \to 0} \lambda(t)f\lambda(t)^{-1}$ exists. This verifies (1). Moreover, since the fiber of $\pi_\lambda$ over $0 \in \bA^1(k)$ is equal to the projection $P_H(\lambda) \to Z_H(\lambda)$, the uniqueness of the limit shows (2) as well.
\end{proof}

Recall from \cite[3.8]{BMRT-instability} that if $k$ is a field, $H$ is a reductive $k$-group, $X$ is a finite type affine $k$-scheme with an action of $H$, and $x \in X(k)$ is a point, then the $H(k)$-orbit $H(k) \cdot x$ is \textit{cocharacter-closed} over $k$ if for every cocharacter $\lambda\colon \bGm \to H$ over $k$ such that the limit $x_\lambda = \lim_{t \to 0} \lambda(t) \cdot x$ exists, we have $x_\lambda \in H(k) \cdot x$. By the Hilbert--Mumford--Kempf theorem \cite[4.4]{Kempf-GIT}, if $k$ is perfect then $H(k) \cdot x$ is cocharacter-closed if and only if $H \cdot x$ is Zariski-closed. The proof of part (1) of the following theorem is very similar to the proof of \cite[5.9]{BMRT-instability}.


\begin{theorem}\label{theorem:closed-orbit}
    Let $k$ be a field, and let $G$ and $H$ be finite type $k$-group schemes. Suppose that $H$ is reductive and that, if $\chara k > 0$, then there is no surjective $k$-homomorphism $G^0 \to \bGa$.
    \begin{enumerate}
        \item The $H(k)$-orbit of a $k$-homomorphism $f\colon G \to H$ is cocharacter-closed in the scheme $\uHom_{\gp{k}}(G, H)_{\rm{red}}$ if and only if $f(G)$ is $H$-cr over $k$.
        \item If $G$ is smooth and $k$ is perfect, then every component of the scheme $\uHom_{\gp{k}}(G, H)_{\rm{red}}$ with a $k$-point contains a unique closed $H^0$-orbit.
    \end{enumerate}
\end{theorem}

\begin{proof}
    By Theorem~\ref{theorem:field-rep} and Corollary~\ref{cor:char-0-reduction}, the component of $\uHom_{\gp{k}}(G, H)_{\rm{red}}$ containing $f$ is affine of finite type, so the term \textit{cocharacter-closed} is well-defined. Beginning with (1), suppose that $f(G)$ is $H$-cr. Let $\lambda\colon \bGm \to H$ be a cocharacter over $k$ such that the limit $f_\lambda \coloneqq \lim_{t \to 0} \lambda(t)f\lambda(t)^{-1}$ exists. By Lemma~\ref{lemma:limits-exist}(1), it follows that $f(G) \subset P_H(\lambda)$. Since $f(G)$ is $H$-cr, there exists an R-Levi $k$-subgroup $L$ of $P_H(\lambda)$ such that $f(G) \subset L$. Thus by \cite[2.5(iii)]{BMRT-instability}, there is some $u \in U_H(\lambda)(k)$ such that $L = u^{-1} Z_H(\lambda) u = Z_H(u^{-1}\lambda u)$. Thus $u^{-1}\lambda u$ centralizes $f(G)$, and we have
    \[
    f = \lim_{t \to 0} u^{-1}\lambda(t)ufu^{-1}\lambda(t)^{-1}u = \lim_{t \to 0} u^{-1} (\lambda(t) u \lambda(t)^{-1}) \lambda(t) f \lambda(t)^{-1} (\lambda(t) u^{-1} \lambda(t)^{-1}) u = u^{-1} f_\lambda u
    \]
    because $\lim_{t \to 0} \lambda(t) u \lambda(t)^{-1} = 1$ by definition of $U_H(\lambda)$. Thus indeed $f_\lambda \in H^0(k) \cdot f$.


    Conversely, suppose that $H(k) \cdot f$ is cocharacter-closed over $k$. By \cite[5.3]{BHMR-cocharacter}, it follows that $H^0(k) \cdot f$ is cocharacter-closed over $k$. Let $\lambda\colon \bGm \to H$ be a cocharacter over $k$ such that $f(G) \subset P_H(\lambda)$. By Lemma~\ref{lemma:limits-exist}(1), the limit $f_\lambda \coloneqq \lim_{t \to 0} \lambda(t) f \lambda(t)^{-1}$ exists, and \cite[3.10]{BMRT-instability} shows that there is some $u \in U_H(\lambda)(k)$ such that $u f_\lambda u^{-1} = f$. By Lemma~\ref{lemma:limits-exist}(2), we have $f_\lambda(G) \subset Z_H(\lambda)$, so $f(G) \subset Z_H(u\lambda u^{-1})$. Thus indeed $G$ is $H$-cr over $k$.


    We move on to (2), so suppose that $G$ is smooth and $k$ is perfect. By Theorem~\ref{theorem:field-rep} and Corollary~\ref{cor:char-0-reduction}, the functor $\sH = \uHom_{\gp{k}}(G, H)_{\rm{red}}$ is representable by a disjoint union of finite type affine $k$-schemes. By Lemma~\ref{lemma:finite-git}, the GIT quotient $\sH/\!/H$ is a disjoint union of \textit{finite} $k$-schemes, so the same is true of $\sH/\!/H^0$. Let $\pi\colon \sH \to \sH/\!/H^0$ be the quotient map. If $C$ is any connected component of $\sH$, then $\pi|_C$ factors through a single point $x$ of $\sH/\!/H^0$. If $C$ contains a $k$-point, then $x$ must be a $k$-point of $\sH/\!/H^0$. By \cite[9.1.4]{Alper-adequate}, if $[\sH/H^0]$ is the stack quotient then the map $[\sH/H^0] \to \sH/\!/H^0$ is an adequate moduli space in the sense of \cite[5.1.1]{Alper-adequate}. The closed points of $[\sH/H^0](\ov{k})$ are in natural bijection with the closed $H^0_{\ov{k}}$-orbits of $\sH(\ov{k})$, so \cite[5.3.1(5)]{Alper-adequate} implies that the set $(\sH/\!/H^0)(\ov{k})$ consists of the closed $H^0_{\ov{k}}$-orbits of $\sH(\ov{k})$, and thus $C_{\ov{k}}$ admits a unique closed $H^0_{\ov{k}}$-orbit. This shows that $C$ admits at most one closed $H^0$-orbit; since $k$ is perfect, \cite[4.4]{Kempf-GIT} shows that some such orbit exists.
\end{proof}

When $G$ and $H$ are both reductive, Theorem~\ref{theorem:closed-orbit}(1) was proven in \cite[2.11]{booher-tang} using \cite[3.7]{BMR05}. The following result shows that the implication also goes the other way.

\begin{cor}\cite[3.7]{BMR05}
    Let $k$ be an algebraically closed field, and let $G$ and $H$ be smooth $k$-group schemes such that $H$ is reductive. If $G$ is topologically generated (i.e., weakly generated) by $x_1, \dots, x_n \in G(k)$ and $f\colon G \to H$ is a $k$-homomorphism, then $f(G)$ is $H$-cr if and only if the $H$-orbit of $(f(x_1), \dots, f(x_n))$ in $H^n$ is closed.
\end{cor}

\begin{proof}
    By Lemma~\ref{lemma:generation-by-subschemes}, the natural morphism $\iota\colon \sH \coloneqq \uHom_{\gp{k}}(G, H) \to H^n$ given by $f \mapsto (f(x_1), \dots, f(x_n))$ is monic and satisfies the valuative criterion of properness. By Lemma~\ref{lemma:tfg}, the hypotheses imply that if $\chara k > 0$ then there is no surjective $k$-homomorphism $G^0 \to \bGa$. Therefore, Theorem~\ref{theorem:field-rep} and Corollary~\ref{cor:char-0-reduction} show that the reduction $\sH_{\rm{red}}$ is a disjoint union of finite type $k$-schemes, and \cite[IV\textsubscript{3}, 8.11.5]{EGA} shows that the restriction of $\iota$ to any connected component of $\sH_{\rm{red}}$ is a closed embedding. In particular, an $H$-orbit $Z$ in $\sH_{\rm{red}}$ is closed if and only if the $H$-orbit $\iota(Z)$ is closed in $H^n$. Thus the result follows from Theorem~\ref{theorem:closed-orbit} and the Hilbert--Mumford--Kempf theorem \cite[4.4]{Kempf-GIT}.
\end{proof}

Using Theorem~\ref{theorem:closed-orbit}, one can adapt many of the results of \cite{BMR05} to non-smooth groups. As a simple example, if $k = \ov{k}$ and $G$ is a closed $k$-subgroup of a reductive $k$-group $H$ which is $H$-cr, then it follows from Theorem~\ref{theorem:closed-orbit} and Matsushima's theorem \cite[A]{Richardson-coset}, the reduced centralizer $Z_H(G)_{\rm{red}}$ is reductive.\footnote{\cite{Richardson-coset} uses classical foundations as in \cite{Borel}, so for him a variety is identified with its set of geometric points; this accounts for the presence of the underlying reduced subscheme in this statement, which is missing from \cite{Richardson-coset}.} 

We now list a couple of simple corollaries of Theorem~\ref{theorem:closed-orbit} which are known, but whose existing proofs are complicated (and rather different). See \cite[3.3]{booher-tang} for another such result.

\begin{cor}\cite[10.3]{Martin}
    Let $K/k$ be an extension of algebraically closed fields, and let $G$ and $H$ be smooth finite type $k$-group schemes such that there is no surjective $k$-homomorphism $G^0 \to \bGa$ and $H$ is reductive. If $f\colon G_K \to H_K$ is a $K$-homomorphism such that $f(G_K)$ is $H_K$-cr, then there exists a $k$-homomorphism $f_0\colon G \to H$ and an element $h \in H(K)$ such that $h (f_0)_K h^{-1} = f$.
\end{cor}

\begin{proof}
    Let $\sH = \uHom_{\gp{k}}(G, H)$, and let $C$ be the connected component of $\sH$ such that $f$ lies in $C(K)$. Let $f_0 \in C(k)$ have closed $H$-orbit. By Theorem~\ref{theorem:closed-orbit}(1), the morphism $f$ has closed $H_K$-orbit, so Theorem~\ref{theorem:closed-orbit}(2) implies that $(f_0)_K$ and $f$ are $H(K)$-conjugate.
\end{proof}

\begin{cor}\cite[3.8]{BMR05}
    Let $H$ be a reductive group over a field $k$. For every fixed integer $N$, there are only finitely many $H$-orbits of $H$-cr closed $k$-subgroup schemes of $H$ of order $N$.
\end{cor}

\begin{proof}
    Since there are only finitely many groups of order $N$, it suffices to show that for any fixed such group $G$ the scheme $\uHom_{\gp{k}}(G, H)$ has only finitely many closed $H$-orbits. The scheme $\uHom_{\gp{k}}(G, H)$ is affine, hence has only finitely many components, and Theorem~\ref{theorem:closed-orbit} shows that each such component contains a \textit{unique} closed $H$-orbit.
\end{proof}




The following corollary does not quite belong in this section, but its proof uses Lemma~\ref{lemma:limits-exist}.

\begin{cor}\label{cor:torus-affine}
    Let $k$ be a field, and let $G$ and $H$ be finite type affine $k$-group schemes such that there is no surjective $k$-homomorphism $G^0 \to \bGa$. If $T$ is a maximal $k$-torus of $G$, then the restriction morphism $r\colon \uHom_{\gp{k}}(G, H) \to \uHom_{\gp{k}}(T, H)$ is affine of finite type.
\end{cor}

\begin{proof}
    We may and do assume $k = \ov{k}$. Suppose first that $G = T \ltimes U$ for a torus $T$ and a unipotent group $U$ in which all weights of $T$ on $U$ lie in the same half space of $X^*(T)$. In other words, suppose that there is a cocharacter $\lambda\colon \bGm \to T$ such that $\langle\lambda, \alpha\rangle > 0$ for all weights $\alpha$ of $T$ on $\Lie U$. By Theorem~\ref{theorem:field-rep}, under either of the above conditions the morphism $r$ realizes $\uHom_{\gp{k}}(G, H)$ as a disjoint union of finite type (relatively) affine $\uHom_{\gp{k}}(T, H)$-schemes. Thus it suffices to show that if $C$ is any connected component of $\uHom_{\gp{k}}(T, H)$ then $r^{-1}(C)$ is connected. There is an evident section $\sigma$ to $r$ in our case, so it suffices to show that $r$ has connected fibers. By Lemma~\ref{lemma:limits-exist}, for any $f\colon G \to H$ the limit $\lim_{t \to 0} f(\lambda(t))\cdot f\cdot f(\lambda(t)^{-1})$ exists, and it is equal to the composition $G \to T \xrightarrow{f} H$. Thus every element in $r^{-1}(C)$ can be connected to the image of $\sigma$ by a map from $\bA_k^1$, proving $r^{-1}(C)$ is connected.

    Now pass to the case of general $G$. Note first that the restriction map $\uHom_{\gp{k}}(G, H) \to \uHom_{\gp{k}}(G^0, H)$ is affine by Lemma~\ref{lemma:passage-to-id-comp}, so we may and do pass from $G$ to $G^0$ to assume that $G$ is connected. Let $\lambda\colon \bGm \to T$ be a cocharacter such that $\langle\lambda, \alpha\rangle \neq 0$ for all (nonzero) roots $\alpha$ of $(G, T)$. By the dynamic method (see especially \cite[2.1.8]{CGP}), there are $\lambda$-stable closed subgroups $U_{G}(-\lambda)$, $Z_{G}(\lambda)$, and $U_{G}(\lambda)$ of $G$ such that $U_{G}(-\lambda)$ and $U_{G}(\lambda)$ are unipotent, $T \subset Z_{G}(\lambda)$ is the unique maximal torus, and the multiplication morphism $U_G(-\lambda) \times Z_G(\lambda) \times U_G(\lambda) \to G$ is an open embedding. Thus the natural maps
    \[
    \uHom_{\gp{k}}(T \ltimes U_G(\pm\lambda), H) \to \uHom_{\gp{k}}(T, H)
    \]
    are affine by the previous paragraph. Note that, if $G_0$ is the closed $k$-subgroup scheme of $G$ generated by $U_G(\pm\lambda)$, then $G_0$ is normalized by $U_G(\pm\lambda)$ and $Z_G(\lambda)$, hence it is normal in $G$. Since $G/(T \cdot G_0)$ is clearly unipotent, it follows from our assumption that $G$ is generated by $T \ltimes U_G(\pm\lambda)$. Applying Lemma~\ref{lemma:strong-generation-by-subschemes}, we find that
    \[
    \uHom_{\gp{k}}(G, H) \to \prod_{\pm} \uHom_{\gp{k}}(T \ltimes U_{G}(\pm\lambda), H)
    \]
    is a closed embedding, concluding the proof.
\end{proof}

\begin{remark}
    When $H = \GL_n$ and $G$ is connected reductive, the fibers of the morphism $r$ in Corollary~\ref{cor:torus-affine} are connected (as one sees by considering Jordan--H\"older factors). For general connected reductive $H$, it would be interesting to describe the components of the fibers of $r$.
\end{remark}

\section{Results over a general base}\label{section:representability}

In this section we investigate the representability of $\uHom_{\gp{S}}(G, H)$ using Theorem~\ref{theorem:general-criterion}. In Section~\ref{section:generic} (resp.\ \ref{section:global}), we investigate representability when $S$ is the spectrum of an artin local ring (resp.\ a general scheme). For a general scheme $S$, it seems difficult to describe optimal conditions on $G$ and $H$ under which $\uHom_{\gp{S}}(G, H)$ is representable, and in Section~\ref{section:examples} we will describe examples not covered by the results of Section~\ref{section:global}.

\subsection{Generic representability}\label{section:generic}

We generalize part of Theorem~\ref{theorem:field-rep} to artin local bases.

\begin{theorem}\label{theorem:generic-rep}
    Let $S$ be an irreducible scheme with generic point $\eta$, and let $G$ be a smooth $S$-group scheme with connected fibers such that $G_\eta$ is affine. The following are equivalent.
    \begin{enumerate}
        \item\label{item:3} There is no nontrivial $k(\eta)$-homomorphism $G_\eta \to \bGa$.
        \item\label{item:2} The functor $\underline{\Hom}_{\gp{k(\overline{\eta})}}(G_{\overline{\eta}}, \bGm)$ is representable.
        \item\label{item:1} There exists a dense open subscheme $U$ of $S$ such that for every $U$-scheme $U'$ and smooth $U'$-affine $U'$-group scheme $H$, the functor $\underline{\Hom}_{\gp{U'}}(G_{U'}, H)$ is representable by a disjoint union of finitely presented $U'$-affine $U'$-schemes.
    \end{enumerate}
\end{theorem}

\begin{proof}
    It is clear that (\ref{item:1}) implies (\ref{item:2}). The fact that (\ref{item:2}) implies (\ref{item:3}) is established in Examples~\ref{example:ga-quotient-0} and \ref{example:ga-quotient-p}. Thus we need only show that (\ref{item:3}) implies (\ref{item:1}). By working locally and spreading out, we may and do assume that $S = \Spec A$ is affine noetherian and $S_{\rm{red}}$ is normal. We are free to shrink $S$ further if necessary (and we will do so later).

By \cite[A.2.11]{CGP}, there exist $k(\eta)$-subtori $T_1^\circ, \dots, T_n^\circ \subset G_\eta$ for which the multiplication morphism $T_1^\circ \times \cdots \times T_n^\circ \to G_\eta$ is dominant. By deformation theory for tori \cite[IX, 3.6]{SGA3II} and spreading out (see especially \cite[IV\textsubscript{3}, 9.6.1(ii)]{EGA}), we may therefore shrink $S$ to assume that there are $S$-subtori $T_1, \dots, T_n \subset G$ such that the multiplication morphism $\mu\colon T_1 \times_S \cdots \times_S T_n \to G$ is dominant on all $S$-fibers. Since $G$ is smooth, generic flatness and the fibral flatness criterion \cite[IV\textsubscript{3}, 11.3.10]{EGA} combine to show that there exists a dense open subscheme $\Omega$ of $G$ such that $\mu^{-1}(\Omega) \to \Omega$ is faithfully flat. By Chevalley's theorem on constructible sets \cite[IV\textsubscript{1}, 1.8.4]{EGA}, the image of $\Omega$ in $S$ is constructible and contains the generic point, so it contains a dense open subscheme. Replacing $S$ with this subscheme, we may and do assume that $\Omega$ is universally schematically dense in $G$. Note that $\mu^{-1}(\Omega)$ is universally schematically dense in $T_1 \times_S \cdots \times_S T_n$ since the latter has connected fibers.

Since $\mu^{-1}(\Omega) \to \Omega$ is faithfully flat, the fiber product $V = \mu^{-1}(\Omega) \times_{\Omega} \mu^{-1}(\Omega)$ is faithfully flat over $S$. By shrinking $S$, we may and do assume that $V$ is pure over $S$ by \cite[I, 3.3.8]{Raynaud-Gruson}. We have therefore verified the hypotheses of Theorem~\ref{theorem:general-criterion}.
\end{proof}

\begin{cor}\label{cor:artin-rep}
    Let $S$ be an artin local scheme, let $G$ be a smooth affine $S$-group scheme, and let $S'$ be an $S$-scheme. If $G_s^0$ admits no nontrivial $k(s)$-homomorphisms to $\bGa$, and $H$ is a smooth $S'$-affine $S'$-group scheme, then $\uHom_{\gp{S'}}(G_{S'}, H)$ is representable by a disjoint union of finitely presented $S'$-affine $S'$-schemes. The restriction map $\uHom_{\gp{S'}}(G_{S'}, H) \to \uHom_{\gp{S'}}(G_{S'}^0, H)$ is affine.
\end{cor}

\begin{proof}
    After passing to an \'etale cover of $S$, we may and do find sections $g_1, \dots, g_n \in G(S)$ which meet every component of $G$. The proof of \cite[2.2]{booher-tang} shows that the induced morphism
    \[
    \uHom_{\gp{S'}}(G_{S'}, H) \to \uHom_{\gp{S'}}(G_{S'}^0, H) \times_{S'} H^n
    \]
    sending $f$ to $(f|_{G_{S'}^0}, f(g_1), \dots, f(g_n))$ is a closed embedding, so we may reduce to the case that $G$ is connected. Thus we may apply Theorem~\ref{theorem:generic-rep} ((\ref{item:3}) $\Rightarrow$ (\ref{item:1})), noting that $S$ is topologically a single point.
\end{proof}

\subsection{Global representability}\label{section:global}

Next, we investigate representability of $\uHom_{\gp{S}}(G, H)$ for a general scheme $S$. As in Section~\ref{ss:field-rep}, we begin with several counterexamples, now of a more ``global" nature: in each of the following examples, $\uHom_{\gp{S'}}(G_{S'}, H_{S'})$ turns out to be representable for every $S$-scheme $S'$ which is topologically a single point.

\begin{example}\label{example:semi-abelian}
    Let $R$ be a DVR, and let $\sA$ be a nontrivial smooth $R$-group scheme such that $G_\eta$ is an abelian variety and $G_s$ is a torus. Then $\uHom_{\gp{R}}(\bGm, \sA)$ is not representable: this follows as in Example~\ref{example:ga-quotient-0} from deformation theory for tori \cite[IX, 3.6]{SGA3II} and the fact that there are no nontrivial homomorphisms from a torus to an abelian variety over a field. The functor $\uHom_{\gp{R}}(\sA, \bGm)$ is also not representable: indeed, \textit{loc.\ cit.}\ and the fibral isomorphism criterion \cite[IV\textsubscript{4}, 17.9.5]{EGA} show that if $\pi$ is a uniformizer of $R$ then $\sA_{R/\pi^n}$ is an $R/\pi^n$-torus for all $n \geq 1$. Thus the same argument applies, using the fact that there are no nontrivial homomorphisms from an abelian variety to a torus over a field.
\end{example}

\begin{example}\label{example:comp-gp-jump}
    Let $R$ be a strictly henselian complete DVR, and let $G$ be a smooth affine $R$-group scheme such that $G^0$ is not closed in $G$.\footnote{Many such $G$ exist, and they are common in Bruhat--Tits theory. For a very explicit example (in which both fibers of $G^0$ are generated by tori), see \cite[XIX, \S 3]{SGA3III}.} Let $\Gamma$ be the component group of $G_s$, and consider $\Gamma$ as a constant $R$-group scheme. Then $\uHom_{\gp{R}}(G, \Gamma)$ is not representable: indeed, if $\fm$ is the maximal ideal of $R$, then for each $n \geq 1$ there is a natural map $f_n\colon G_{R/\fm^n} \to \Gamma$ with kernel $G^0_{R/\fm^n}$. However, there is no $R$-homomorphism $f\colon G \to \Gamma$ lifting $(f_n)$: indeed, we would then have $\ov{G^0} \subset \ker f$, so $\ker(f_1)$ would \textit{strictly} contain $G^0_{R/\fm}$, a contradiction.
\end{example}

\begin{example}\label{example:purity}
    Let $R$ be a strictly henselian DVR with residue field $k$ and fraction field $K$, and let $G$ be a quasi-finite separated \'etale $R$-group scheme which is not finite. (In other words, $G$ is not \textit{pure}.) Note that $G(k)$ can be identified with a proper subgroup of $G(K)$ using the identity $G(k) = G(R)$ and the inclusion $G(R) \subset G(K)$. By passing to an extension of $R$, we may and do assume $G(K) = G(\ov{K})$.

    We claim that there is a finite constant $R$-group scheme $H$ such that $\sH = \uHom_{\gp{R}}(G, H)$ is not representable. Note that $\sH$ is formally \'etale and satisfies the valuative criterion of properness, so if it were representable then it would be finite \'etale. Since $R$ is strictly henselian, it would then follow that the natural map $\sigma\colon \sH(K) = \sH(R) \to \sH(k)$ is bijective. However, by the construction of \cite{Linderholm} (see also \cite[p.\ 21]{Eilenberg-Moore}), there exists a finite (abstract) group $H$ and a homomorphism $f\colon G(k) \to H$ admitting two distinct extensions to a homomorphism $G(K) \to H$. Thus $\sigma$ is \textit{not} injective.
\end{example}

\begin{example}\label{example:heisenberg}
    Let $k$ be a field of characteristic $p > 0$. Let $U_0$ be the standard Heisenberg group over $k$; in other words, $U_0$ is scheme-theoretically isomorphic to $\bA_k^3$ with multiplication given by
    \[
    (x, y, z)(x', y', z') = (x + x', y + y', z + z' + x'y).
    \]
    Let $U$ be a ``twisted" Heisenberg group over $k[\![t]\!]$: as a $k$-scheme, $U$ is $\bA_{k[\![t]\!]}^3$, and the multiplication is given by
    \[
    (x, y, z)(x', y', z') = (x + x', y + y', z + z' + x'y - t^p x'^p y^p).
    \]
    There are natural actions of $\bGm$ on $U_0$ and $U$, with weight $1$ on the first coordinate of $\bA_k^3$, weight $-1$ on the second coordinate, and weight $0$ on the third. With these actions, let $G = \bGm \ltimes U$ and $H = \bGm \ltimes U_0$, both considered as $k[\![t]\!]$-group schemes. Using Lemma~\ref{lemma:gen-by-tori-field}, one may check that $G$ and $H$ have fibers which are generated by tori. However, $\sH$ is \textit{not} representable. To show this, we proceed as in Example~\ref{example:ga-quotient-0}.

    For each $n \geq 0$, let $A_n = k[t]/(t^{p^n})$, and define $\vp_n\colon \bGa \to \bGa$ over $A_n$ by
    \[
    \vp_n(z) = z + t^p z^p + \cdots + t^{p^{n-1}} z^{p^{n-1}}.
    \]
    Using this, define $f_n\colon G_{A_n} \to H_{A_n}$ by
    \[
    f_n(c, x, y, z) = (c, x, y, \vp_n(z)).
    \]
    It is straightforward to check that $f_n$ is a $k$-homomorphism and $(f_n)_n$ defines a nontrivial element of $\varprojlim_n \sH(A_n)$. However, the degree of $\vp_n$ is unbounded as $n \to \infty$, so $(f_n)$ does \textit{not} arise from an element of $\sH(k[\![t]\!])$.

    For a similar (but contrasting) example, see Example~\ref{example:good-heisenberg}.
\end{example}

We move on now to positive results. We will use the notion of \textit{open cells} from the introduction.

\begin{example}\label{example:open-cells}
    Many common classes of smooth affine group schemes admit open cells.
    \begin{enumerate}
        \item If $T$ is a torus acting on $\bGa^n$ with nonzero weights, then $T \ltimes \bGa^n$ admits an open cell.
        \item If $G$ is a split reductive group scheme, then $G$ admits an open cell by \cite[XXII, 4.1.2]{SGA3III}.
        \item More generally, if $P$ is a parabolic subgroup scheme of a split reductive group scheme, then $P$ admits an open cell by \cite[XXII, 5.4.4]{SGA3III}.
        \item If $S$ is the spectrum of a strictly henselian DVR and $\sP$ is a parahoric group scheme of a split reductive group scheme $G$, then $\sP$ admits an open cell by \cite[8.3.14]{Kaletha-Prasad}.
        \item If an $S$-group scheme $G$ admits an open cell, then the fibral rank of $G$ is constant over $S$. For example, if $A \to A'$ is a finite ramified extension of DVRs and $H$ is a reductive group scheme over $A'$, then the Weil restriction $\mathrm{R}_{A'/A}(H)$ does not admit an open cell.
    \end{enumerate}
\end{example}

\begin{theorem}\label{theorem:global-rep}
    Let $S$ be a scheme, and let $G$ be a flat finitely presented $S$-group scheme for which there exists an fpqc cover $\{S'_i \to S\}$ such that each $G_{S'_i}^0$ admits an open cell. Suppose also that $G^0$ is closed in $G$ and $G/G^0$ is a finite \'etale $S$-group scheme. If $H$ is a smooth $S$-affine $S$-group scheme, then $\underline{\Hom}_{\gp{S}}(G, H)$ is representable by an $S$-ind-quasi-affine $S$-scheme, locally of finite presentation and the restriction map $\uHom_{\gp{S}}(G, H) \to \uHom_{\gp{S}}(G^0, H)$ is affine.
    
    If $S_{\rm{red}}$ is normal and either locally noetherian or qcqs, and $G$ admits an open cell (over $S$), then $\underline{\Hom}_{\gp{S}}(G, H)$ is representable by a disjoint union of finitely presented $S$-affine $S$-schemes.
\end{theorem}

\begin{proof}
    Assume first that $G$ has connected fibers over $S$. We may and do pass to an fpqc cover of $S$ (using \cite[\href{https://stacks.math.columbia.edu/tag/0APK}{Tag 0APK}]{stacks-project}) to assume that $G$ admits an open cell. Working locally, we may assume that $S$ is affine, and in particular $S$ is an inverse limit of noetherian affine schemes. Arguing as in Lemma~\ref{lemma:ind-scheme} using \cite[IV\textsubscript{3}, 8.6.3, 8.8.2, 11.2.6; IV\textsubscript{4}, 17.7.8]{EGA}, we may therefore reduce to the case that $S$ is noetherian. Fix $T$ and $U_i$ as in the definition, and let $\Omega_0 \subset G$ be the open image of the multiplication morphism $T \times_S \prod_{i=1}^n U_i \to G$. We may and do assume that $T$ is split.

    Define $S$-subtori $T_0, T_1, \dots, T_n \subset G$ as follows: first, let $T_0 = T$. Since $T$ is split, we may choose a monic cocharacter $\lambda\colon \bGm \to T$ such that $\lambda(\bGm)$ acts nontrivially on each $U_i$, and choose $S$-isomorphisms $u_i\colon \bGa \to U_i$. For each $i$, let $m_i$ be the weight with which $\lambda(\bGm)$ acts on $U_i$. Define $T_i = u_i(1)\lambda(\bGm)u_i(-1)$, so $T_i$ is an $S$-subtorus of $G$. Define a cocharacter $\lambda_i\colon \bGm \to T_i$ by $\lambda_i(x) \coloneqq u_i(1)\lambda(x)u_i(-1)$. If $t$ is a local section of $\bGm$, then we have
    \[
    \lambda_i(t) = \lambda(t)u_i(t^{-m_i} - 1).
    \]
    Thus a simple calculation shows that the multiplication morphism $\mu\colon T_0 \times_S \cdots \times_S T_n \to G$ is given in coordinates by
    \[
    \mu(t_0, \lambda_1(x_1), \dots, \lambda_n(x_n)) = (t_0 \lambda(x_1) \cdots \lambda(x_n)) \cdot u_1(x_n^{-m_1} \cdots x_2^{-m_1}(x_1^{-m_1} - 1)) \cdots u_n(x_n^{m_n} - 1).
    \]
    In particular, $\mu$ factors through $\Omega_0$.
    
    Note that $\mu$ is quasi-finite, hence flat by the fibral flatness criterion \cite[IV\textsubscript{3}, 11.3.10]{EGA} and Miracle Flatness \cite[23.1]{Matsumura}. In particular, the image $\Omega$ of $\mu$ in $G$ is open and fiberwise dense, and it is easy to see that the map $\mu^{-1}(\Omega) \to \Omega$ has constant fiber degree $\prod_{i=1}^n m_i$. By \cite[II, 1.19]{Deligne-Rapoport}, we see that $\mu^{-1}(\Omega) \to \Omega$ is finite flat. By Lemma~\ref{lemma:purity-open} we know that $\Omega$ is pure, so $\mu^{-1}(\Omega) \times_U \mu^{-1}(\Omega)$ is pure. Thus the result follows from Theorem~\ref{theorem:general-criterion}.

    Now consider the general case, dropping the assumption that $G$ has connected fibers. By descent, we may and do pass to an \'etale cover of $S$ to assume that there exist sections $g_1, \dots, g_n \in G(S)$ such that the morphism $(g_1, \dots, g_n)\colon S^{\bigsqcup n} \to G/G^0$ is an isomorphism. Thus $G$ is strongly generated by $G^0$ and $g_1, \dots, g_n$, so the result follows from Lemma~\ref{lemma:strong-generation-by-subschemes}.
\end{proof}

\subsection{Functoriality}

In this section we collect two results which illustrate the behavior of Hom schemes under passage to maximal tori and parabolics.

The following lemma is proven in \cite[XXIV, 7.2.1]{SGA3III} under stronger hypotheses (namely, $G$ must be reductive in \textit{loc.\ cit.}), but also with a slightly stronger conclusion.

\begin{lemma}\label{lemma:affine-to-torus}
    Under the notation and hypotheses of Theorem~\ref{theorem:global-rep}, if $T$ is a fiberwise maximal $S$-torus of $G$, then the restriction morphism $r\colon \uHom_{\gp{S}}(G, H) \to \uHom_{\gp{S}}(T, H)$ is affine.
\end{lemma}

\begin{proof}
    By working fpqc-locally, we may assume $G$ admits an open cell $T \times_S \prod_{i=1}^n U_i$ such that $U_i \cong (\bGa)_S$ for all $i$ and $T$ acts nontrivially on $U_i$. Working locally on $S$ and spreading out, we may further assume that $S$ is excellent (e.g., finite type over $\bZ$) and connected. By Theorem~\ref{theorem:global-rep}, the restriction map $\uHom_{\gp{S}}(G, H) \to \uHom_{\gp{S}}(G^0, H)$ is finitely presented and affine, so we may assume that $G$ has connected fibers. It is enough to show that if $\{U_i\}$ is an affine open cover of $\uHom_{\gp{S}}(T, H)$, then $r^{-1}(U_i)$ is affine for all $i$. If $\widetilde{S} \to S$ is the normalization (which is finite by excellence of $S$), then by a theorem of Chevalley \cite[II, 6.7.1]{EGA}\footnote{A hypothesis of this reference is that $r^{-1}(U_i)$ is noetherian, which we do not yet know. However, we know that $r^{-1}(U_i)$ is locally noetherian, and if $r^{-1}(U_i) \times_S \widetilde{S}$ is affine then it is in particular quasi-compact. Since $r^{-1}(U_i) \times \widetilde{S} \to r^{-1}(U_i)$ is surjective, this immediately implies that $r^{-1}(U_i)$ is quasi-compact, hence noetherian.} it is enough to show that $r^{-1}(U_i) \times_S \widetilde{S}$ is affine; thus we may pass from $S$ to $\widetilde{S}$ to assume that $S$ is normal.
    
    If $n = 1$, i.e., $G = T \ltimes \bGa$ for an $S$-torus $T$ acting with nonzero weights on $\bGa$, then the result follows in precisely the same way as in the proof of Corollary~\ref{cor:torus-affine}, using Theorem~\ref{theorem:global-rep} in place of Theorem~\ref{theorem:field-rep}. In general, the natural map
    \[
    \uHom_{\gp{S}}(G, H) \to \prod_{i=1}^n \uHom_{\gp{S}}(T \ltimes U_i, H)
    \]
    is a closed embedding by Lemma~\ref{lemma:strong-generation-by-subschemes}, and we conclude.
\end{proof}

\begin{lemma}\label{lemma:parabolic-restriction}
    Let $S$ be a scheme, let $G$ be a reductive $S$-group scheme, let $P \subset G$ be a parabolic $S$-subgroup scheme, and let $H$ be a smooth $S$-affine $S$-group scheme. The natural restriction morphism $r\colon \uHom_{\gp{S}}(G, H) \to \uHom_{\gp{S}}(P, H)$ is a finitely presented open embedding.
\end{lemma}

\begin{proof}
    By spreading out (arguing as in the proof of Lemma~\ref{lemma:ind-scheme} and using \cite[3.1.11]{Conrad} and \cite[IV\textsubscript{3}, 8.10.5; IV\textsubscript{4}, 17.7.8]{EGA}) we may and do assume $S$ is locally noetherian. First, $r$ is monic: if $f_1, f_2\colon G \to H$ are $S$-homomorphisms having the same restriction to $P$, then the morphism $f = f_1 f_2^{-1}$ factors through $G/P$. Since $H$ is smooth and $S$-affine, it follows that $f$ factors through a section $S \to H$, and since $f(1) = 1$ we see that $f_1 = f_2$. By Theorem~\ref{theorem:global-rep}, the morphism $r$ is locally of finite presentation. Moreover, $r$ is \'etale: to show this, we must verify the infinitesimal criterion. Thus we may and do assume $S = \Spec A$ is artin local. Let $I \subset A$ be a square zero ideal, and suppose we are given an $A/I$-homomorphism $\ov{f}\colon G_{A/I} \to H_{A/I}$ whose restriction to $P_{A/I}$ extends to an $A$-homomorphism $f_0\colon P \to H$. By \cite[III, 2.1(i), 2.3]{SGA3I}, the obstruction to lifting $\ov{f}$ to an $A$-homomorphism $G \to H$ lies in $\rm{H}^2(G_s, \Lie H_s \otimes_{k(s)} I)$. Moreover, the obstruction is evidently functorial, so to show that a lift $f_1$ of $\ov{f}$ exists (perhaps not extending $f_0$), it suffices to show that the natural map $\rm{H}^2(G_s, \Lie H_s \otimes_{k(s)} I) \to \rm{H}^2(P_s, \Lie H_s \otimes_{k(s)} I)$ is bijective. This follows from \cite[II, Corollary 4.7 c)]{Jantzen}. Finally, the set of extensions of $\ov{f}$ up to $\ker(H(A) \to H(A/I))$-conjugacy is a principal homogeneous space under $\rm{H}^1(G_s, \Lie H_s \otimes_{k(s)} I)$ by \cite[III, 2.1(ii), 2.3]{SGA3I}. This structure of principal homogeneous space is again functorial, so to show that $f_0$ can be extended to a homomorphism $G \to H$ it is enough to check that $\rm{H}^1(G_s, \Lie H_s \otimes_{k(s)} I) \to \rm{H}^1(P_s, \Lie H_s \otimes_{k(s)} I)$ is bijective. This follows again from \cite[II, Corollary 4.7 c)]{Jantzen}.
\end{proof}

\begin{remark}
    If $\pi\colon G' \to G$ is a surjection of reductive $S$-group schemes, then it is natural to ask about properties of the map $\pi^*\colon \uHom_{\gp{S}}(G, H) \to \uHom_{\gp{S}}(G', H)$. With some work, Lemma~\ref{lemma:brion}(\ref{item:brion-2}) can be used to show that $\pi^*$ is always a closed embedding. We can also show that $\pi^*$ is an \textit{open} embedding provided that for each $s \in S$, either $\chara k(s) \neq 2$, or $\ker \pi_s$ is smooth, or $G_{\ov{s}}$ has no simple factors of types $\rm{B}$ or $\rm{C}$. The proof of this latter fact is somewhat involved and will appear in future work of the author.
\end{remark}

\subsection{Further examples}\label{section:examples}

The hypotheses on $G$ in Theorem~\ref{theorem:global-rep} are not optimal, but it is not clear to the author how to formulate a clean statement which is more general (except for the hyper-general Theorem~\ref{theorem:general-criterion}, whose hypotheses are nontrivial to check). In this section we give a couple of examples going beyond Theorem~\ref{theorem:global-rep} in order to further illustrate the method.

\begin{example}\label{example:moy-prasad}
    Let $R$ be a DVR with uniformizer $\pi$, and let $G_0$ be a semisimple $R$-group scheme. For a given integer $n \geq 1$, let $G_n$ be the smooth affine $R$-group scheme such that, for every \textit{flat} $R$-algebra $R'$, we have $G_n(R') = \ker(G_0(R') \to G_0(R'/\pi^n))$ (see \cite[A.5.13]{Kaletha-Prasad}; $G_n$ is the $n$th step in the Moy-Prasad filtration). In particular, $G_n$ has generic fiber $(G_0)_\eta$ and special fiber $(\Lie G_0)_s$ (considered as a vector group). The group $G_0$ acts on $G_n$ by conjugation, and we let $G = G_0 \ltimes G_n$. Note that $G_s$ and $G_\eta$ are both generated by tori, but the fibral rank of $G$ is not constant over $\Spec R$, so the hypotheses of Theorem~\ref{theorem:global-rep} do not hold. Nonetheless, we will verify the hypotheses of Theorem~\ref{theorem:general-criterion} to show that $\uHom_{\gp{S}}(G_S, H)$ is representable for every $R$-scheme $S$ and every smooth $S$-affine $S$-group scheme $H$.

    By extending $R$, we may and do assume that there exists a split maximal subtorus $T_0 \subset G_0$. Let $\Phi$ be the root system of $(G_0, T_0)$, let $\alpha_1, \dots, \alpha_m$ be a system of simple roots for $\Phi$, and let $U_{\alpha, 0}$ be the root group of $G_0$ corresponding to $\alpha \in \Phi$. Moreover, let $T_n$ and $U_{\alpha, n}$ be the $R$-subgroup schemes of $G_n$ corresponding to $T_0$ and $U_{\alpha, 0}$, respectively. Let $(\bGm)_n$ denote the similarly-defined $R$-group scheme corresponding to $\bGm$. Note that there is a dense open subscheme
    \[
    \Omega = T_0 \times_R \prod_{\alpha \in \Phi} U_{\alpha, 0} \times_R \prod_{\alpha \in \Phi} U_{\alpha, n} \times_R T_n \subset G,
    \]
    where $\Phi$ is given an arbitrary (fixed) order. For each $\alpha$, let $u_{\alpha, 0}\colon \bGa \to U_{\alpha, 0}$ and $u_{\alpha, n}\colon \bGa \to U_{\alpha, n}$ be isomorphisms with the property that
    \begin{equation}\label{equation:lang}
    u_{\alpha, n}\left(\frac{x}{1+\pi^n x}\right)u_{\alpha, 0}\left(\frac{1}{1+\pi^n x}\right)u_{-\alpha, n}(x)u_{\alpha, 0}\left(\frac{-1}{1 + \pi^n x}\right)u_{-\alpha, n}(-x(1 + \pi^n x)) = (\alpha^\vee)_n(1 + \pi^n x),
    \end{equation}
    where $(\alpha^\vee)_n\colon (\bGm)_n \to T_n$ is the map induced by the coroot $\bGm \to T_0$ corresponding to $\alpha$. The fact that such $u_{\alpha, 0}$ and $u_{\alpha, n}$ may be found follows from a calculation with $\SL_2$, explained in the proof of \cite[XIII, 8.1]{Lang}; explicitly, if $G = \SL_2$ then we define
    \[
    u_{\alpha, 0}(x) = \begin{pmatrix}
        1 &x \\ 0 &1
    \end{pmatrix}, u_{\alpha, n}(x) = \begin{pmatrix}
        1 &\pi^n x \\ 0 &1
    \end{pmatrix}, u_{-\alpha, n}(x) = \begin{pmatrix}
        1 &0 \\ \pi^n x &1
    \end{pmatrix}, (\alpha^\vee)_n(y) = \begin{pmatrix}
        y &0 \\ 0 &y^{-1}
    \end{pmatrix}
    \]
    functorially on points valued in flat $R$-algebras.
    
    Choose a cocharacter $\lambda\colon \bGm \to T_0$ such that $\lambda(\bGm)$ acts nontrivially on $U_{\alpha, 0}$ and $U_{\alpha, n}$ for all $\alpha \in \Phi$; let $m_\alpha$ be the corresponding weights. Define $R$-subtori $T_{\alpha, 0}$ and $T_{\alpha, n}$ of $G$ via
    \[
    T_{\alpha, 0} = u_{\alpha, 0}(1)\lambda(t)u_{\alpha, 0}(-1) \text{ and } T_{\alpha, n} = u_{\alpha, n}(1)\lambda(t)u_{\alpha, n}(-1)
    \]
    As in the proof of Theorem~\ref{theorem:global-rep}, we have
    \[
    T_{\alpha, 0} = \{\lambda(t)u_{\alpha, 0}(t^{-m_{\alpha}} - 1)\colon t \in \bGm\} \text{ and } T_{\alpha, n} = \{\lambda(t)u_{\alpha, n}(t^{-m_{\alpha}} - 1)\colon t \in \bGm\}.
    \]
    A calculation as in the proof of Theorem~\ref{theorem:global-rep} shows that the multiplication morphism $\mu\colon T_0 \times_R \prod_{\alpha \in \Phi} T_{\alpha, 0} \times_R \prod_{\alpha \in \Phi} T_{\alpha, n} \to G$ factors through the closed subscheme
    \[
    \Omega' = T_0 \times_R \prod_{\alpha \in \Phi} U_{\alpha, 0} \times_R \prod_{\alpha \in \Phi} U_{\alpha, n} \times_R \{1\} \subset \Omega,
    \]
    and the factored map is quasi-finite and flat of constant fibral degree.
    
    For each $1 \leq j \leq m$, define the subfunctor $W_j$ of $(\Omega')^3$ to consist of tuples
    \[
    \left(u_{\alpha_j, n}\left(\frac{x}{1 + \pi^n x}\right), u_{\alpha_j, 0}\left(\frac{1}{1 + \pi^n x}\right)u_{-\alpha_j, n}(x), u_{\alpha_j, 0}\left(\frac{-1}{1 + \pi^n x}\right)u_{-\alpha_j, n}(-x(1 + \pi^n x))\right),
    \]
    where $x$ ranges over (functorial) points of the open subscheme $Y$ of $\bA^1$ defined by the invertibility of $1 + \pi^nx$. Note that the map $Y \to \bGm$, $x \mapsto \frac{1}{1+\pi^nx}$, is an isomorphism of $R$-schemes, so the map $Y \to U_{\alpha_j, 0}$, $x \mapsto u_{\alpha_j, 0}(\frac{1}{1+\pi^nx})$ is a locally closed embedding. By cancellation (e.g., \cite[I, 2.5.6 b)]{DG} applied to the projection of $(\Omega')^3$ onto the second $U_{\alpha_j, 0}$-factor), it follows that the map $Y \to \Omega'$ whose image defines $W_j$ is a locally closed embedding, i.e., $W_j$ is a subscheme of $(\Omega')^3$.
    
    Define subschemes $X_1, \dots, X_m$ of $\left(T_0 \times_R \prod_{\alpha \in \Phi} T_{\alpha, 0} \times_R \prod_{\alpha \in \Phi} T_{\alpha, n}\right)^3$ via $X_j = (\mu^3)^{-1}(W_j)$. The map $\mu^3\colon X_j \to W_j$ is quasi-finite and flat of constant fibral degree by base change, using the corresponding property of $\mu$. Moreover, by (\ref{equation:lang}) the map $m_G\colon W_j \to G$ factors through an isomorphism to $(\alpha_j^\vee)_n((\bGm)_n)$, so the natural map $m_G \circ \mu^3|_{X_j}$ factors through a quasi-finite flat map to $(\alpha_j^\vee)_n((\bGm)_n)$ with constant fibral degree. In particular, $X_j$ is fppf and pure. Moreover, the multiplication morphism
    \[
    T_0 \times_R \prod_{\alpha \in \Phi} T_{\alpha, 0} \times_R \prod_{\alpha \in \Phi} T_{\alpha, n} \times_R \prod_{j=1}^m X_j \to \Omega
    \]
    is quasi-finite and flat of constant fibral degree, so it is finite over a dense open of $\Omega$ by \cite[II, 1.19]{Deligne-Rapoport}, and the hypotheses of Theorem~\ref{theorem:general-criterion} are satisfied.

\end{example}

\begin{example}\label{example:good-heisenberg}
In Example~\ref{example:heisenberg}, the form of the multiplication map for $U$ is important. To be more precise, let $S$ be an $\bF_p$-scheme and let $U$ be the $S$-scheme $\bA_S^3$ with group law
\[
(x, y, z)(x', y', z') = \left(x + x', y + y', z + z' + \sum_{i=0}^n a_i x'^{p^i}y^{p^i}\right)
\]
where $a_i \in \Gamma(S, \sO)$ for all $i$, and $a_n$ is a \textit{unit}; this last condition fails in Example~\ref{example:heisenberg}. We briefly sketch a proof that $\uHom_{\gp{S}}(\bGm \ltimes U, H)$ is representable for all smooth $S$-affine $S$-group schemes $H$ (although $\bGm \ltimes U$ does not satisfy the hypotheses of Theorem~\ref{theorem:global-rep}). The details are rather similar to those in Example~\ref{example:moy-prasad}, so they will not be given in full. By working locally and spreading out, we may and do assume that $S$ is locally noetherian.

Let $T_1 = (1, 0, 0)\bGm(-1, 0, 0)$ and let $T_2 = (0, 1, 0)\bGm(0, -1, 0)$ be $S$-subtori of $G = \bGm \ltimes U$. If $\lambda\colon \bGm \to G$ is the natural cocharacter, then we compute
\[
T_1 = \{(\lambda(t), t^{-1} - 1, 0, 0)\colon t \in \bGm\} \text{ and } T_2 = \{(\lambda(t), 0, t - 1, 0)\colon t \in \bGm\}
\]
as in the proof of Theorem~\ref{theorem:global-rep}. The multiplication morphism $\mu\colon \bGm \times_S T_1 \times_S T_2 \to G$ factors through an open embedding to $\bGm \times (\bA^1 - \{-1\})^2$.

By descent, we may assume that there is a global unit $c \in \Gamma(S, \sO_S^\times)$ which does not take the value $\pm 1$. If $C$ denotes the central $\bGa$ in $U$, then the map $\vp\colon \bGa \to C$,
\[
\vp(x) = (c, 0, 0)(0, x, 0)(-c, 0, 0)(0, -x, 0) = \left(0, 0, \sum_{i=0}^n a_i (-cx)^{p^i}\right)
\]
is a finite surjective homomorphism. Similarly to Example~\ref{example:moy-prasad}, let
\[
W = \{((c, 0, 0), (0, x, 0), (-c, 0, 0), (0, -x, 0))\colon x \in \bA^1 - \{\pm 1\}\} \subset U^4
\]
and let $X = (\mu^4)^{-1}(W)$, so $X$ is fppf and pure because $c$ and $-c$ factor through $\bA^1 - \{-1\}$ (and thus $X \cong W$ by the above). Using the tori $\bGm, T_1, T_2$ and the scheme $X$, we conclude representability using Theorem~\ref{theorem:general-criterion} (taking $\Omega = \bGm \times (\bA^1 - \{-1\})^2 \times \bA^1$).
\end{example}

In \cite[XXIV, \S 7.4]{SGA3III}, the scheme $\uHom_{\gp{\bZ}}(\SL_2, \SL_2)$ is described explicitly using detailed calculations with $2 \times 2$ matrices. The following example generalizes this without any calculations.

\begin{example}\label{example:sga3-endo}
    Let $G$ be a split simple $\bZ$-group scheme, and fix a pinning of $G$. The scheme $\uHom_{\gp{\bZ}}(G, G)$ is the disjoint union of the following schemes:
    \begin{enumerate}
        \item $\Spec \bZ$, corresponding to the trivial homomorphism $G \to G$,
        \item $\Aut_{G/\bZ}$,
        \item\label{item:frobenii} for each prime number $p$ and each isogeny $\pi\colon G_{\bF_p} \to G_{\bF_p}$ compatible with the pinning of $G$ and inducing a $p$-morphism of root data (see \cite[XXI, 6.8.1]{SGA3III}, and require $q \neq 1$), a scheme isomorphic to $\Aut_{G_{\bF_p}/\bF_p}$, corresponding to $\Aut_{G_{\bF_p}/\bF_p}$-translates of $\pi$.
    \end{enumerate}
    To see this, note first that the inclusion $\Spec \bZ = \uHom_{\gp{\bZ}}(G, 1) \to \uHom_{\gp{\bZ}}(G, G)$ corresponding to the trivial homomorphism is a clopen embedding by Lemma~\ref{lemma:identity-clopen}. If $S$ is a connected scheme and $f\colon G_S \to G_S$ is a nontrivial $S$-homomorphism, then $\ker f_s$ is finite for every $s \in S$ and thus $f$ is an isogeny by the fibral flatness criterion \cite[IV\textsubscript{3}, 11.3.10]{EGA}. Any two pinnings of $G$ are conjugate by an inner $S$-automorphism: this follows from \cite[3.2.6, 5.2.13(2)]{Conrad}, and the fact that the roots form a basis for the character lattice of any maximal torus in the adjoint group $G/Z(G)$. Thus we may translate $f$ by an $S$-automorphism of $G_S$ to assume that $f$ is compatible with the chosen pinning in the sense of \cite[XXIII, 1.3]{SGA3III}. By \cite[XXIII, 4.1]{SGA3III}, $f$ arises via base change either from an isogeny $G \to G$ or from an isogeny $G_{\bF_p} \to G_{\bF_p}$. This yields the desired description.

    For most pairs $(G, p)$ of groups $G$ and primes $p$, the morphisms (\ref{item:frobenii}) are all powers of the Frobenius $F_{G_{\bF_p}}$. However, for the pairs $(\rm{B}_2, 2)$, $(\rm{F}_4, 2)$, and $(\rm{G}_2, 3)$, there are extra isogenies; see for example \cite[XXI, 7.5.2]{SGA3III}.
\end{example}

\appendix

\section{Weak generation}\label{section:weak-generation}

In \cite[Question 8.6.9]{Furter-Kraft}, the following question is raised.

\begin{question}\label{question:furter-kraft}
    If $k$ is an algebraically closed field of characteristic $0$ and we are given linear algebraic groups $G$ and $L$ over $k$ and closed $k$-subgroups $H$ and $K$ of $G$ which generate $G$, then is the natural morphism
    \begin{equation}\label{equation:furter-kraft}
    \uHom_{\gp{k}}(G, L) \to \uHom_{\gp{k}}(H, L) \times \uHom_{\gp{k}}(K, L)
    \end{equation}
    a closed embedding?\footnote{In \cite{Furter-Kraft}, the authors work in the language of ind-varieties, and thus they ``really" ask whether (\ref{equation:furter-kraft}) becomes a closed embedding when all functors involved are restricted to \textit{reduced} $k$-algebras.}
\end{question}

The word ``generate" is not defined in \cite{Furter-Kraft}, so we will answer Question~\ref{question:furter-kraft} under several different interpretations of this word. If ``generate" is interpreted in the strong sense of Definition~\ref{def:generation}, then the answer is positive by Lemma~\ref{lemma:strong-generation-by-subschemes} (regardless of $\chara k$). The following example shows that the answer is negative in general if ``generate" is interpreted in the weak sense.

\begin{example}\label{example:furter-kraft-weak}
    Let $G = \bGm \rtimes \bZ/2$, where $\bZ/2$ acts on $\bGm$ by inversion. Let $t \in k^\times$ be an element of infinite order, so $t$ weakly generates $\bGm$. Let $H$ and $K$ be the closed $k$-subgroups of $G$ generated by $(t, 1)$ and $(1, 1)$ respectively, so $H$ and $K$ are both isomorphic to $\bZ/2$. If we set $L = G$, then the natural map (\ref{equation:furter-kraft}) is not a closed embedding: indeed, the right side is of finite type while the left side is not.
\end{example}

We will show that the answer to Question~\ref{question:furter-kraft} is positive when $H(k)$ and $K(k)$ generate $G(k)$ as an abstract group. First, we will show that the answer is positive if $G$ is \textit{connected} and $H$ and $K$ weakly generate $G$. The following example shows that the assumption $\chara k = 0$ is essential.

\begin{example}\label{example:furter-kraft-char-p}
    Let $G = L = \SL_2$ over a field $k$ of characteristic $p > 0$, and suppose that there exists $t \in k$ which is transcendental over the prime field. Let $H$ and $K$ be the closed $k$-subgroup schemes of $G$ defined by
    \[
    H = \left\{\begin{pmatrix} 1 & a + bt^{-1} \\ 0 & 1 \end{pmatrix}\colon a, b \in \bF_p\right\}
    \]
    and
    \[
    K = \left\{\begin{pmatrix} 1 & 0 \\ a(t-1) + bt(t-1) & 1 \end{pmatrix}\colon a, b \in \bF_p\right\}.
    \]
    Note that $H$ and $K$ are both isomorphic to $\bZ/p^2$. The equation
    \[
    \begin{pmatrix}
        1&1\\0&1
    \end{pmatrix}
    \begin{pmatrix}
        1&0\\t-1&1
    \end{pmatrix}
    \begin{pmatrix}
        1&-t^{-1}\\0&1
    \end{pmatrix}
    \begin{pmatrix}
        1&0\\-t(t-1)&1
    \end{pmatrix}
    = \begin{pmatrix}
        t&0\\0&t^{-1}
    \end{pmatrix}
    \]
    shows that the closed $k$-subgroup scheme $G_0$ of $G$ generated by $H$ and $K$ contains the diagonal torus $T$ of $G$. Since $G_0$ intersects both root groups for $T$, we have $G_0 = G$ by a scaling argument. Note that $\uHom_{\gp{k}}(H, L)$ and $\uHom_{\gp{k}}(K, L)$ are both of finite type, but $\uHom_{\gp{k}}(G, L)$ is not (see Example~\ref{example:sga3-endo}), so (\ref{equation:furter-kraft}) is not a closed embedding in this case.
\end{example}

Our first lemma is similar to Lemma~\ref{lemma:strong-generation-by-subschemes} and generalizes \cite[2.10]{booher-tang}.

\begin{lemma}\label{lemma:generation-by-subschemes}
    Let $S$ be a locally noetherian scheme, and let $G$ and $H$ be $S$-group schemes such that $G$ is finite type and flat and $H$ is separated. Let $\{f_i\colon X_i \to G\}_{i \in I}$ weakly generate $G$. The natural morphism $\iota\colon \uHom_{\gp{S}}(G, H) \to \prod_{i \in I} \uHom_{\sch{S}}(X_i, H)$ is monic and satisfies the valuative criterion of properness.
\end{lemma}

\begin{proof}
    We follow the proof of \cite[2.10]{booher-tang} rather closely. First, $\iota$ is monic: let $S'$ be an $S$-scheme, and suppose that $\vp, \psi\colon G_{S'}\to H_{S'}$ are two $S'$-group homomorphisms such that $\vp \circ (f_i)_{S'} = \psi \circ (f_i)_{S'}$ for all $i$. Note that the equalizer $E$ of $\vp$ and $\psi$ is a closed subscheme of $G_{S'}$ through which each $(f_i)_{S'}$ factors, so by assumption $E = G_{S'}$ and $\vp = \psi$.
    
    For the valuative criterion, let $R$ be a DVR over $S$ with fraction field $K$. Let $\vp_i\colon (X_i)_R \to H_R$ be $R$-scheme morphisms, and suppose that there exists a $K$-group homomorphism $\vp\colon G_K \to H_K$ such that $\vp \circ (f_i)_K = \vp_i|_{(X_i)_K}$ for all $i$. Let $\Gamma$ be the schematic closure of the graph of $\vp$ in $G_R \times H_R$, so $\Gamma$ is a flat finite type $R$-group scheme, the first projection $\pi\colon \Gamma \to G_R$ is an $R$-homomorphism, and $((f_i)_R, \vp_i)\colon (X_i)_R \to G_R \times H_R$ factors through $\Gamma$ for all $i$ (as can be checked over $K$ since each $X_i$ is flat). To verify the valuative criterion of properness, it is enough to show that $\pi$ is an isomorphism of $R$-group schemes. Note that each $(f_i)_R$ factors through $\pi$, and it follows that $\pi_s$ is weakly schematically dominant. Since $\Gamma$ is of finite type, the schematic image of $\pi_s$ is a closed subscheme in $G_s$ by \cite[VI\textsubscript{B}, 1.2]{SGA3I}, and by weak schematic dominance the schematic image is equal to $G_s$. By definition of the schematic image, this implies that $\pi_s$ is schematically dominant and thus $\pi_s$ is faithfully flat by \cite[V, 3.3.1]{Perrin} (whose proof simplifies considerably when the groups are of finite type). By the fibral flatness criterion \cite[IV\textsubscript{3}, 11.3.10]{EGA}, it follows that $\pi$ is flat. Since $\ker \pi$ has trivial generic fiber, it also has trivial special fiber and thus $\pi_s$ is a closed embedding by \cite[VI\textsubscript{B}, 1.4.2]{SGA3I}. It follows that $\pi_s$ is an isomorphism. We conclude by the fibral isomorphism criterion \cite[IV\textsubscript{4}, 17.9.5]{EGA}.
\end{proof}

By Lemma~\ref{lemma:generation-by-subschemes} and \cite[IV\textsubscript{3}, 8.11.5]{EGA}, the map (\ref{equation:furter-kraft}) is a closed embedding if and only if it is representable and of finite type. To verify these conditions when $H$ and $K$ weakly generate $G$, we begin with a series of lemmas which are of independent interest.

If $G$ is a finite type $k$-group scheme, define the derived group $\sD(G)$ of $G$ to be the intersection of all normal closed $k$-subgroup schemes $N$ of $G$ for which $G/N$ is commutative. It is standard to check that $\sD(G)$ is a normal closed $k$-subgroup scheme of $G$ and $G/\sD(G)$ is commutative. If $G$ is smooth, then this is the usual derived group.

\begin{lemma}\label{lemma:derived-surjection}
    Let $k$ be a field, and let $G$ be a smooth connected $k$-group scheme. There is no surjective $k$-homomorphism $\sD(G) \to \bGm$.
\end{lemma}

\begin{proof}
    We may and do assume $k = \ov{k}$. Lemma~\ref{lemma:surjection-affine-parts} shows that $G = G_{\rm{aff}} \cdot G_{\rm{ant}}$, and $G_{\rm{ant}}$ is central in $G$ by \cite[\S III.3.8]{DG}. Thus $\sD(G) = \sD(G_{\rm{aff}})$, so we may and do assume that $G$ is affine. Since $G/\sR_u(G)$ is reductive, we see that $\sD(G)$ is strongly generated by unipotent subgroups, whence the claim.
\end{proof}

The following lemma is a version of the Burnside basis theorem from finite group theory.

\begin{lemma}\label{lemma:burnside-basis}
    Let $k$ be a field, and let $U$ be a nilpotent finite type $k$-group scheme. If $V \subset U$ is a closed $k$-subgroup scheme for which $V \to U/\sD(U)$ is faithfully flat, then $V = U$.
\end{lemma}

\begin{proof}
    Let $1 = Z_0 \subset Z_1 \subset \cdots \subset Z_n = U$ be the upper central series for $U$, i.e., $Z_{i+1} = Z(U/Z_i)$ for all $i$. If $n = 1$, then $U$ is abelian and thus it is clear that $V = U$. In general, induction shows that $Z_1 \cdot V = U$, so in particular $V$ is normal in $U$. Now $U/V$ is nilpotent, so if it is nontrivial then it has a nontrivial commutative quotient $Q$. By commutativity, $Q$ is also a quotient of $U/\sD(U)$, so the map $V \to U/V \to Q$ is faithfully flat, hence trivial.
\end{proof}

Recall from Section~\ref{section:groups} that, if $G$ is a finite type affine $k$-group scheme, then $G_{\rm{t}}$ denotes the closed $k$-subgroup of $G$ generated by tori.

\begin{lemma}\label{lemma:strong-generation}
    Let $k$ be a field, and let $G$ be a smooth connected $k$-group scheme. If $H_1, \dots, H_n$ are closed $k$-subgroup schemes of $G$ which weakly generate $G$, then $H_1, \dots, H_n$, and $\sD(G_{\rm{t}})$ strongly generate $G$.
\end{lemma}

\begin{proof}
    Clearly $\sD(H_{\ov{k}}) \subset \sD(H)_{\ov{k}}$ for any finite type $k$-group scheme $H$, so we may and do assume $k = \ov{k}$.\footnote{In fact, a limit argument shows that $\sD(H_{\ov{k}}) = \sD(H)_{\ov{k}}$, but we do not need this.} Passing from $G$ to $G/\sD(G_{\rm{t}})$, we need only show that the $H_i$ strongly generate $G$ when $\sD(G_{\rm{t}})$ is trivial. In this case, $G_{\rm{t}}$ is commutative, hence a torus. Every action of a connected group scheme on a torus is trivial (since tori have \'etale Aut schemes), so because $G/G_{\rm{t}}$ is unipotent by \cite[A.2.11]{CGP} we have $G_{\rm{aff}} = T \times U$ for a torus $T$ and a smooth connected unipotent $k$-group scheme $U$. In particular, $G_{\rm{aff}}/\sD(U)$ is commutative. By \cite[\S III.3.8]{DG}, the group $G_{\rm{ant}}$ is also commutative. By Lemma~\ref{lemma:surjection-affine-parts}, we have $G = G_{\rm{aff}} \cdot G_{\rm{ant}}$. Every action of a connected group scheme on an anti-affine group scheme is trivial (since anti-affine group schemes have \'etale Aut schemes by \cite[5.3]{Brion-hom}), so it follows that $G/\sD(U)$ is commutative. Since the images of the $H_i$ in $G/\sD(U)$ generate, commutativity of $G/\sD(U)$ implies that $H_1 \cdots H_n \cdot \sD(U) = G$. For dimension and smoothness reasons, we see $(H_1)_{\rm{red}}^0 \cdots (H_n)_{\rm{red}}^0 \cdot \sD(U) = G$. If $G_0$ is the smallest closed $k$-subgroup scheme of $G$ containing each $(H_i)_{\rm{red}}^0$, then $G_0$ is strongly generated by the $(H_i)_{\rm{red}}^0$ by \cite[VI\textsubscript{B}, 7.4]{SGA3I}. By Lemma~\ref{lemma:burnside-basis}, since $G_0 \to G/\sD(U)$ is surjective we find $G_0 = G$, as desired.
\end{proof}

\begin{prop}\label{prop:furter-kraft-weak}
    Let $k$ be a field of characteristic $0$, and let $G$ and $L$ be connected finite type $k$-group schemes. If $H_1, \dots, H_n$ are closed $k$-subgroup schemes of $G$ which weakly generate $G$, then
    \begin{equation}\label{equation:furter-kraft-2}
        \uHom_{\gp{k}}(G, L) \to \prod_{i=1}^n \uHom_{\gp{k}}(H_i, L)
    \end{equation}
    is a closed embedding.
\end{prop}

\begin{proof}
    By Lemmas~\ref{lemma:strong-generation} and \ref{lemma:strong-generation-by-subschemes}, the natural morphisms
    \[
    r_1\colon \uHom_{\gp{k}}(G, L) \to \uHom_{\gp{k}}(G_{\rm{t}}, L) \times \prod_{i=1}^n \uHom_{\gp{k}}(H_i, L)
    \]
    and
    \[
    r_2\colon \uHom_{\gp{k}}(G, L) \to \uHom_{\gp{k}}(\sD(G_{\rm{t}}), L) \times \prod_{i=1}^n \uHom_{\gp{k}}(H_i, L)
    \]
    are closed embeddings. By Theorem~\ref{theorem:field-rep}, the functor $\uHom_{\gp{k}}(G_{\rm{t}}, L)$ is a locally finite type $k$-scheme, so the fact that $r_1$ is a closed embedding implies that (\ref{equation:furter-kraft-2}) is representable and locally of finite type. By Corollary~\ref{cor:char-0-reduction} and Lemma~\ref{lemma:derived-surjection}, the functor $\uHom_{\gp{k}}(\sD(G_{\rm{t}}), L)_{\rm{red}}$ is an affine $k$-scheme. Thus  the fact that $r_2$ is a closed embedding implies that (\ref{equation:furter-kraft-2}) is also quasi-compact, hence of finite type. But then Lemma~\ref{lemma:generation-by-subschemes} combines with \cite[IV\textsubscript{3}, 8.11.5]{EGA} to show the result.
\end{proof}

Note that Proposition~\ref{prop:furter-kraft-weak} answers Question~\ref{question:furter-kraft} in a strong way when $G$ is connected, but we have not yet answered the question under the assumption that $H(k)$ and $K(k)$ generate $G(k)$ as an abstract group; Example~\ref{example:furter-kraft-weak} does not furnish a negative answer to this question.

\begin{prop}\label{prop:furter-kraft}
    Let $k$ be an algebraically closed field, and let $G$ and $L$ be smooth finite type $k$-group schemes. If $H$ and $K$ are closed $k$-subgroup schemes of $G$ such that $H(k)$ and $K(k)$ generate $G(k)$ as an abstract group, then $H^0$ and $K^0$ weakly generate $G^0$. If moreover $\chara k = 0$, then (\ref{equation:furter-kraft}) is a closed embedding.
\end{prop}

\begin{proof}
    Let $M$ be the closed $k$-subgroup scheme of $G^0$ weakly generated by $H^0$ and $K^0$. Let $k_0 \subset k$ be a finitely generated subfield of $k$ over which $G$, $H$, and $K$ (and hence $M$) are defined. Increasing $k_0$, we may further assume that the maps $H(k_0) \to H(k)/H^0(k)$ and $K(k_0) \to K(k)/K^0(k)$ are surjective. If $X$ is the subgroup of $G(k)$ generated by $H(k)$ and $K(k)$, then the image of $X$ in $G(k)/M(k)$ is therefore contained in $(G/M)(k_0)$. If $G/M$ is positive-dimensional, then $(G/M)(k_0) \neq (G/M)(k)$: this follows, for example, from Noether normalization. Thus if $M \neq G^0$ then $X \neq G(k)$, proving the first claim.

    For the final claim, note that $H^0$, $K^0$, and $\sD((G^0)_{\rm{t}})$ strongly generate $G^0$ by the previous paragraph and Lemma~\ref{lemma:strong-generation}. Consequently, $H$, $K$, and $\sD((G^0)_{\rm{t}})$ strongly generate $G$, so the argument in the proof of Proposition~\ref{prop:furter-kraft-weak} shows the result.
\end{proof}

\bibliography{bibliography}

\end{document}